\documentclass[a4paper,11pt]{amsart}
\usepackage{amsfonts}
\usepackage{amssymb}
\usepackage[utf8]{inputenc}
\usepackage{amsmath}
\usepackage{pdflscape}
\usepackage{graphicx}
\providecommand{\U}[1]{\protect\rule{.1in}{.1in}}
\usepackage[colorlinks=true, linkcolor=red, citecolor=blue]{hyperref}
\usepackage[]{epsfig}
\usepackage[]{pstricks}
\usepackage{tikz}

\newtheorem*{defin*}{Definition}
\newtheorem{teo}{Theorem}[section]
\newtheorem{proposition}{Proposition}[section]

\newtheorem{cor}{Corollary}[section]
\newtheorem{defin}{Definition}[section]
\newtheorem{rem}{Remark}
\newtheorem{remarks}{Remarks}

\numberwithin{equation}{section}
\makeatletter
\@namedef{subjclassname@2020}{\textup{2020} Mathematics Subject Classification}
\makeatother

\begin{document}

	\pagenumbering{arabic}	
\title[Control for Fractional Schr\"odinger equation]{The fractional Schr\"odinger equation on compact manifolds: Global controllability results}
\author[Capistrano-Filho]{Roberto de A. Capistrano-Filho*}
\address{Departamento de Matem\'atica,  Universidade Federal de Pernambuco (UFPE), 50740-545, Recife (PE), Brazil.}
\email{roberto.capistranofilho@ufpe.br}
\author[Pampu]{Ademir B. Pampu}
\address{Departamento de Matem\'atica,  Universidade Federal de Pernambuco (UFPE), 50740-545, Recife (PE), Brazil.}
\email{ademir\_bp@hotmail.com}
\thanks{*Corresponding author: roberto.capistranofilho@ufpe.br}
\subjclass[2020]{35Q55; 93B05; 93D15; 35A21; 35R11; 35S05} 
\keywords{Global controllability, Global stabilization, Fractional Schr\"odinger equation, Pseudo-differential calculus}

\begin{abstract}
The goal of this work is to prove global controllability and stabilization properties for the fractional Schr\"odinger equation on $d$-dimensional compact Riemannian manifolds without boundary $(M,g)$. To prove our main results we use techniques of  pseudo-differential calculus on manifolds. More precisely, by using microlocal analysis, we are able to prove propagation of regularity which together with the so-called \textit{Geometric Control Condition} and \textit{Unique Continuation Property} help us to prove global control results for the system under consideration. As a main novelty this manuscript presents the relation between the geometric control condition and the controllability for the fractional Schr\"odinger equation.
\end{abstract}
\maketitle

\section{Introduction}
\subsection{Presentation of the model} 
The fractional Schr\"odinger equation is a fundamental equation of fractional quantum mechanics. It was discovered by Laskin \cite{Laskin2000,Laskin2002} as a result of extending the Feynman path integral, from the Brownian-like to L\'evy-like quantum mechanical paths. A path integral over the L\'evy-like quantum-mechanical paths results in a generalization of quantum mechanics. Additionally, the fractional Schr\"odinger equation plays an important role in the study of controllability problems for three dimensional water wave systems, as it is shown in \cite{Zhu}.

Our work will treat the global controllability and stabilization properties of the generalized fractional Schr\"odinger equation on a compact Riemannian manifold $(M, g)$ without boundary, namely
\begin{equation}\label{prob-intr}
\begin{cases}
i \partial_{t} u+\Lambda^{\sigma}_{g}u=0,& \quad \text{ on } M\times ]0, T[,
\\ u(x,0)=u_{0}(x),& \quad  x\in M,
\end{cases}
\end{equation}
where $\sigma\in[2,\infty)$, $M$ is a compact Riemannian manifold with dimension $d<[\sigma] +1$ and $\Lambda_{g}^{\sigma}$ is defined by $\left(\sqrt{-\Delta_{g}}\right)^{\sigma}$. Here, $\Delta_{g}$ is the Laplace–Beltrami operator associated to the metric $g$ and $[\sigma]$ is the integer part of $\sigma$.

It is important to point out that equation \eqref{prob-intr} can be seen as a generalization of the Schr\"odinger equation when $\sigma=2$ (see e.g. \cite{Cazenave2003,Tao}) or of the fourth-order Schr\"odinger equation if $\sigma=4$ (see, for instance, \cite{Pausader2007,Pausader2009}).

\subsection{Setting of the problem}\label{setting} As mentioned before, our goal is to study global properties of stabilization and, consequently, controllability for the generalized nonlinear fractional Schr\"odinger equation on a compact Riemannian manifold. More precisely, we deal with the following system 
\begin{equation}\label{prob-introa}
\begin{cases}
i \partial_{t} u+\Lambda^{\sigma}_{g}u +P'(|u|^{2})u=0,&\quad \text{ on } M\times \mathbb{R}_{+},
\\ u(x,0)=u_0(x),&\quad x\in M.
\end{cases}
\end{equation}
Here, we consider $P$ as a polynomial function with real coefficients and $P'$ its derivative.

In order to determine if the system \eqref{prob-introa} is controllable in large time for a control supported in a small open subset of $M$, we will study equation \eqref{prob-introa} from a control point of view with a forcing term $h=h(x,t)$ added to the equation as a control input 
\begin{equation}\label{prob-intro_aa}
\begin{cases}
i \partial_{t} u+\Lambda^{\sigma}_{g}u +P'(|u|^{2})u=h,&\quad \text{ on } M\times \mathbb{R}_{+},
\\ u(x,0)=u_0(x),&\quad x\in M,
\end{cases}
\end{equation}
where $h$ is assumed to be supported in  a given open subset $\omega \subset M$. Therefore,  the following classical issues  related with the control theory are considered in this work:

\vglue 0.2 cm
\noindent\textbf{Exact control problem:}  Given an initial state $u_0$ and a terminal state $v_{0}$, in a certain space, can one find an appropriate control input $h$ such that equation \eqref{prob-intro_aa} admits a solution $u$ satisfying $u(\cdot,0)= u_0$ and $u(\cdot,T)= v_{0}$?

\vglue 0.2 cm
\noindent\textbf{Stabilization problem:}  Can one find a feedback control law $h=h(x,t)$ such that system \eqref{prob-intro_aa} is asymptotically stable as $t\to\infty $?

\subsection{State of the art} 
When \eqref{prob-intro_aa} is considered on a general compact Riemannian manifold $M$ is not of our knowledge any result about control theory. However, there are interesting results in different domains when one considers $\sigma=2$ or $\sigma=4$ on $\Lambda_{g}^{\sigma}=\left(\sqrt{-\Delta_{g}}\right)^{\sigma}$,  namely, nonlinear Schr\"odinger equation (NLS) and fourth-order nonlinear Schr\"odinger equation (4NLS), respectively. We will give  a brief state of the arts to the reader, precisely, we will present a sample of the control results for both cases.

\subsubsection{Previous results for NLS} 
In what concerns the NLS equation on general compact Riemannian manifolds, that is, system \eqref{prob-intro_aa} with $\sigma = 2$, the first difficulty arise in proving 
the uniform well-posedness since the Strichartz estimates does not hold globally in time and it has a loss of derivatives when compared with the euclidean case. However, this issue is now well understood since the interesting article due to Burq \textit{et al.}
\cite{Burq-Gerard-Tz} (for an excellent review of this topic the authors suggest that the reader see \cite{Burq-Gerard-Tz} and the references therein).

Taking advantage of the Strichartz estimates and propagation results of microlocal defect measures it was proved by Dehman \textit{et al.} \cite{dehman-gerard-lebeau} that the celebrated \textit{Geometric Control Condition} (GCC) was sufficient to get controllability results for the NLS on two dimensional compact Riemannian manifolds. More precisely, in \cite{dehman-gerard-lebeau}, the authors showed that the following assumptions are sufficient to get exact controllability result for solutions of \eqref{prob-intro_aa} when
$\sigma = 2$: 
\begin{itemize}
       \item[(A)] The control function $h$ is effective on an open set $\omega$ which geometrically controls \(M\); i.e. there exists \(T_{0}>0,\) such that every geodesic of \(M\)
traveling with speed 1 and issued at \(t=0,\) enters the set \(\omega\) in a time \(t<T_{0}\) .
       
       \item[(B)] For every \(T>0,\) the only solution lying in the space \(C([0, T], H^{1}(M) )\) of the system 
       \begin{equation*}
       \begin{cases}
i \partial_{t} u+\Delta_{g}u+b_{1}(x,t) u+b_{2}(x,t) \overline{u} = 0,& \quad \text{ on } M\times(0,T),\\
u=0,& \quad \text{ on } \omega\times( 0, T) ,\end{cases}
\end{equation*}
where $b_{1}, b_{2}  \in L^{\infty}\left(0, T, L^{p}(M)\right)$ for some \(p>0\), is the trivial
one \(u \equiv 0\).
\end{itemize}
\begin{rem}
Assumption $(A)$ is the so called geometric control condition (see e.g. \cite{BaLeRa92}) and $(B)$ is a unique continuation condition.  It is worth mentioning that as provided by Maci\`a in \cite{Macia2,Macia}, under certain geometric assumptions, $(A)$ is equivalent to assumption $(B)$ when the potential does not depend on the time variable, that is, $V(x,t) = V(x)$, and recently those results where expanded to the fractional Schr\"odinger equation in \cite{Macia3}.
\end{rem}

Considering compact Riemannian manifolds of dimension $d \geq 3$, Strichartz estimates does not yield uniform well posedness result at the energy level for the NLS equation, property which seems to be very important to prove controllability results. In this way, Burq \textit{et al.} in two works \cite{Burq1,Burq2} managed to introduce the Bourgain spaces $X^{s, b}$ on certain manifolds without boundary where the bilinear Strichartz estimates can be showed and, consequently, they get the uniform well-posedness for the NLS equation. Taking advantage of these results, Laurent \cite{Laurent-siam} proved that (GCC)  is sufficient to prove the exact controllability for the NLS in $X^{s, b}$ spaces on three-dimensional Riemannian compact manifolds. Similarly of this work, we mention \cite{RoZha} and \cite{Laurent-esaim} where controllability results were studied for the NLS in euclidean and periodic domains, respectively, relying on the properties of the Bourgain spaces. Finally, observe that in the following three works \cite{Burq3,Jaffard,KoLo} the authors showed that although (GCC) is sufficient to get controllability results to the NLS equation, it is not necessary.

\subsubsection{ Previous results for 4NLS}
Considering  $\sigma=4$, let us denote the fourth-order Schr\"odinger equation as follows
\begin{equation}
\label{fourthlin}
i\partial_tu +\Delta^2u=0.
\end{equation}
There are interesting results for equation \eqref{fourthlin} in the sense of control problems in a bounded domain of $\mathbb{R}$ or $\mathbb{R}^n$ and, more recently, on a periodic domain $\mathbb{T}$ which we will summarize below.

The first result about the exact controllability of the linearized fourth order Schr\"odinger equation \eqref{fourthlin} on a bounded domain $\Omega$ of $\mathbb{R}^n$ is due to Zheng and  Zhongcheng in \cite{zz}. In this work, by means of an $L^2$--Neumann boundary control, the authors proved that the solution is exactly controllable in $H^s(\Omega)$, $s=-2$, for an arbitrarily small time. They used Hilbert Uniqueness Method (HUM) (see, for instance, \cite{DolRus1977,lions1}) combined with the multiplier techniques to get the main result of the article. More recently,  in \cite{zheng}, Zheng proved another interesting problem related with the control theory. To do this, he showed a global Carleman estimate for the fourth order Schr\"odinger equation posed on a finite domain. The Carleman estimate is used to prove the Lipschitz stability for an inverse problem consisting in retrieving a stationary potential in the Schr\"odinger equation from boundary measurements.

Still on control theory Wen \textit{et al.} in two works \cite{WenChaiGuo1,WenChaiGuo}, studied well-posedness and control problems
related with the equation \eqref{fourthlin} on a bounded domain of $\mathbb{R}^n$, for $n\geq2$. In \cite{WenChaiGuo1}, they considered the Neumann boundary controllability with collocated observation. With this result in hand, the exponential stability of the closed-loop system under proportional output feedback control holds. Recently, the same authors, in \cite{WenChaiGuo}, gave  positive answers when considered the equation with hinged boundary by either moment or Dirichlet boundary control and collocated observation, respectively.

To get a general outline of the control theory already done for the system \eqref{fourthlin}, two interesting problems were studied recently by Aksas and Rebiai \cite{AkReSa} and Gao \cite{Peng}: Uniform stabilization and stochastic control problem, in a smooth bounded domain $\Omega$ of $\mathbb{R}^n$ and on the interval $I=(0,1)$ of $\mathbb{R}$, respectively. In the first work, by introducing suitable dissipative boundary conditions, the authors proved that the solution decays exponentially in $L^2(\Omega)$ when the damping term is effective on a neighborhood of a part of the boundary. The results are established by using multiplier techniques and compactness/uniqueness arguments. Regarding the second work, above mentioned, the author showed Carleman estimates for forward and backward stochastic fourth order Schr\"odinger equations which provided the proof of the observability inequality,  unique continuation property and, consequently,  the exact controllability for the forward and backward stochastic system associated with \eqref{fourthlin}.

Lastly, in \cite{CaCa}, the first author showed the global stabilization and exact controllability  properties of the fourth order nonlinear Schr\"odinger system 
\begin{equation}\label{fourthC}
\begin{cases}
i\partial_{t}u +\partial_{x}^{2}u-\partial_x^4u
=\lambda |u|^2u +f,& (x,t)\in \mathbb{T}\times \mathbb{R},\\
u(x,0)=u_0(x),& x\in \mathbb{T},\end{cases}
\end{equation}
on a periodic domain $\mathbb{T}$ with internal control supported on an arbitrary sub-domain of $\mathbb{T}$. More precisely, by certain properties of propagation of compactness and regularity in Bourgain spaces, for the solution of the associated linear system, the authors proved that system \eqref{fourthC} is globally exponentially stabilizable. This property together with the local exact controllability ensures that 4NLS is globally exactly controllable on $\mathbb{T}$.

\subsubsection{Additional comments of Schrödinger-like equation}
It is interesting to mention that if we consider the following control problem in $L^{2}\left(\mathbb{T}^{2}\right)$:
\begin{equation}\label{D}
\partial_{t} u+i|D|^{\sigma} u= \varphi F, \quad 1 \leq \sigma < 2,
\end{equation}
where $\varphi \in C^{\infty}\left(\mathbb{T}^{2}\right)$, $|D|^{\sigma}$ is the fractional laplacian defined in $\mathbb{T}^{2}$ and  $F \in C([0, T], L^{2}(\mathbb{T}^{2}))$ is complex valued, we still have some interesting results in the literature. When $\sigma=1$, system \eqref{D} is so-called \textit{(half) wave equation}. We have that system \eqref{D} is exactly controllable in the sense defined in section \ref{setting}, if and only if the control function $\varphi F$ is supported in a region where the (GCC) is satisfied, and the proof can be adapted using the ideas of \cite{BaLeRa92,Burq-Gerard,dehman-Rousseau-Leautaud}. 

Still regarding the control problems to the system \eqref{D}, to finish this small sample of state of the art, in an interesting article, Zhu \cite{Zhu} studied the exact controllability for spatially periodic water waves with surface tension, by localized exterior pressures applied to free surfaces.  He showed that in any dimension, the exact controllability for \eqref{D} holds within an arbitrarily short time, for sufficiently small and regular data, provided that the region of control satisfies the (GCC). Additionally, was obtained that in the middle of the two typical cases $1<\alpha<2$, (GCC) is necessary to prove the exactly controllability of \eqref{D} on $\mathbb{T}^{2}$.

\subsection{Main results}
Let us introduce the issues addressed  in this work. We want to study the stabilization and exact controllability for the generalized nonlinear fractional Schr\"odinger equation on compact Riemannian manifolds
\begin{equation}\label{prob-intro}
\begin{cases}
i \partial_{t} u+\Lambda^{\sigma}_{g}u +P'(|u|^{2})u=0,& \quad \text{on } M\times ] 0, T[,
\\ u(x,0)=u_0(x),&\quad x\in M.
\end{cases}
\end{equation}
Remember that $M$ is a compact Riemannian manifold without boundary of dimension $d<[\sigma] +1$, $\Lambda_{g}^{\sigma}$ is defined by $\left(\sqrt{-\Delta_{g}}\right)^{\sigma}$, with $\sigma\in[2,\infty)$, and $\Delta_{g}$ is the Laplace–Beltrami operator associated to the metric $g$. 

In \eqref{prob-intro}, $P$ is a polynomial function with real coefficients and $P'$ its derivative satisfying the following two properties
\begin{equation}\label{p1}
P(0) = 0\quad \text{ and }  \quad P(r) \to \infty \text{ as }r \to \infty
    \end{equation}
    and 
    \begin{equation}\label{p2}
P'(r) \geq C > 0, \quad \text{ for every} \quad r\geq0.
    \end{equation}
The condition \eqref{p1} means that the nonlinear term in \eqref{prob-intro} is defocusing and once we have it we can assume \eqref{p2} without loss of generality. Indeed, changing the unknown function $u(\cdot,t)$ in \eqref{prob-intro} by $e^{i \lambda t}u(\cdot,t)$, the corresponding problem has the  nonlinear term $P'$ replaced by $P' + C$.

Note that for  $u_{0}\in H^{\frac{\sigma}{2}}(M)$ system \eqref{prob-intro} admits a unique solution $u \in C([0,+\infty), H^{\frac{\sigma}{2}}(M)$). This solution satisfies some integrability properties and Strichartz estimates which will be detailed in the next section. Additionally, equation \eqref{prob-intro} displays two energy levels, namely:  $L^2$ energy (or mass) and the nonlinear energy or $H^{\frac{\sigma}{2}}-$energy, given by
\begin{eqnarray}\label{Func-En}
	E(t) = \int_{M}{|\Lambda^{\frac{\sigma}{2}}_{g}u(t)|^{2}}dx + \int_{M}{P(|u(t)|^{2})}dx.
 \end{eqnarray}
 
Our first result concerns global stabilization. To introduce the problem, let $a = a(x) \in C^{\infty}(M)$ be a real valued nonnegative function and consider the system
\begin{equation}\label{prob-diss1}
                \begin{cases}
i \partial_{t} u  + \Lambda^{\sigma}_{g}u  + P'(|u|^{2})u - a(x)(1 - \Delta_{g})^{-\frac{\sigma}{2}}a(x)\partial_{t}u =0, & \text{ on } M\times \mathbb{R}_{+}, \\
u(x,0)=u_0(x),& x\in M.
        \end{cases}
\end{equation}
Also, pick $\omega$  an open subset of $M$ and consider the two following assumptions, that were already mentioned earlier:

\vspace{0.1cm}
\noindent\textbf{$(\mathcal{A})$} \quad \(\omega\) geometrically controls \(M\), i.e there exists \(T_{0}>0\) , such that every geodesic of \(M,\) travelling with speed 1 and issued at \(t=0\) , enters the set \(\omega\) at a time \(t<T_{0}\) .

\vspace{0.1cm}

\noindent\textbf{$(\mathcal{B})$} \quad  For every \(T>0\), the only solution lying in the space $C( ]0,T[, H^{\frac{\sigma}{2}}(M))$ of the system
        \begin{equation*}
        \begin{cases}
        i \partial_{t} u+\Lambda^{\sigma}_{g}u + b_{1}(t, x) u+b_{2}(t, x) \overline{u}=0,& \text{ on } M\times ] 0, T[ ,\\
        u=0,&  \text{ on } \omega\times ] 0, T[,
        \end{cases}
\end{equation*}
where $b_{1}, b_{2} \in C^{\infty}([0, T] \times M)$, is the trivial one $u \equiv 0$.

In what follows, $\omega$ will be related to a cut-off nonnegative function $a = a(x)\in C^{\infty}(M)$ (whose existence is guaranteed by the Whitney theorem) taking real values and such
that
\begin{equation}\label{damped}
\omega= \{x \in M : a(x) \neq 0\}.
\end{equation}

Therefore, our strategy is first to  prove  that system \eqref{prob-diss1} is well-posed in $C([0,+\infty), H^{\frac{\sigma}{2}}(M))$. With this in hand, it is easy to check that its unique solution $u = u(x,t)$ satisfies the energy identity:
\begin{equation}\label{ene}
E(t_{2})- E(t_{1})=-2\int_{t_{1}}^{t_{2}}{\|(1 - \Delta_{g})^{-\frac{\sigma}{2}}a(x)\partial_{\tau}u(\tau) \|_{L^{2}(M)}^{2}}d\tau, \quad\forall t_{2} \geq t_{1} \geq0.
\end{equation}
Observe that for $a(x)=0$ we have, by \eqref{ene}, the mass of the system is indeed conserved. However, assuming the condition \eqref{damped} on some nonempty open set $\omega$ of $M$, identity \eqref{ene} states that we have a possibility of an exponential decay of the solutions related of \eqref{prob-diss1}. In fact, by using the techniques of microlocal analysis, the result that we are able to prove, for large data, can be read as follows:

\begin{teo}[Stabilization] \label{teo-stab}
Let us consider the assumptions $(\mathcal{A})$, $(\mathcal{B})$ and \eqref{damped} as  above. For every $R_0>0$, there exist two constants $C := C(R_{0})>0$ and $\gamma>0$ such that the inequality
\begin{equation}\label{exp}
\|u(t)\|_{H^{\frac{\sigma}{2}}} \leq C e^{-\gamma t}\left\|u_{0}\right\|_{H^{\frac{\sigma}{2}}}, t \in \mathbb{R}_{+}
\end{equation}
holds for every solution $u=u(x,t)$ of the damped system \eqref{prob-diss1} with initial data $u_0$ satisfying $$\|u_{0}\|_{H^{\frac{\sigma}{2}}}\leq R_0.$$
\end{teo}

Now, to give an answer to the global control problem, we need first to prove a local exact controllability result and to combine it with a global stabilization (Theorem \ref{teo-stab}) of the solutions to get the global controllability of the following system
\begin{equation}\label{Ncontrolsys-int} 
\begin{cases}
i \partial_{t} u  + \Lambda^{\sigma}_{g}u  + P'(|u|^{2})u = h(x,t),&   \text{ on } M\times ] 0, T[, \\
u(x,0)=u_0(x),& x\in M.
\end{cases}
\end{equation}
Thus, in this spirit, we will prove the control property of \eqref{Ncontrolsys-int} near to $0$, that will be proved using a perturbation argument introduced by Zuazua in \cite{Zuazua}. To be precise, we will show the following local controllability result:

\begin{teo}[Local controllability]\label{main1}
Let $\omega \subset M$ be an open set satisfying assumption $ (\mathcal{A})$. There exists $\epsilon>0$ such that for any  $u_{0}\in H^{\frac{\sigma}{2}}(M)$  with
\begin{eqnarray}
\|u_{0}\|_{H^{\frac{\sigma}{2}}(M)} \leq \epsilon,
\end{eqnarray} 
one can find a control input $h(x,t):=h \in C([0, T]; H^{\frac{\sigma}{2}}(M))$, with support in $[0, T] \times \omega$, such that 
the unique solution $u \in C([0, T]; H^{\frac{\sigma}{2}}(M))$ of the system \eqref{Ncontrolsys-int}   satisfies $u(\cdot,T) = 0$. 
\end{teo}

Finally, with Theorems \ref{teo-stab}  and \ref{main1} in hand, the following  global exact controllability result for the generalized fractional Schr\"odinger equation is established.

\begin{teo}[Global controllability]\label{controle}
Let $\omega \subset M$ be an open set satisfying assumptions $(\mathcal{A})$ and $(\mathcal{B})$. Then, for every $R_{0} > 0$, 
there exists $T(R_0):=T > 0$ such that for every data 
$u_{0}$ and $v_{0}$ in $H^{\frac{\sigma}{2}}(M)$  satisfying 
\begin{eqnarray}
\|u_{0}\|_{H^{\frac{\sigma}{2}}(M)} \leq R_{0} \quad \text{ and } \quad \|v_{0}\|_{H^{\frac{\sigma}{2}}(M)} \leq R_{0},
\end{eqnarray} 
there exists a control $h \in C([0, T]; H^{\frac{\sigma}{2}}(M))$, with support in $[0, T] \times \omega$, such that 
the unique solution $u \in C([0, T]; H^{\frac{\sigma}{2}}(M))$ of the system 
\begin{equation} \label{prob-diss12}
\begin{cases}
i \partial_{t} u  + \Lambda^{\sigma}_{g}u  + P'(|u|^{2})u =  h,&   \text{ on } M\times ] 0, T[, \\
u(x,0)=u_0(x),& x\in M,
\end{cases}
\end{equation}
 satisfies $u(\cdot,T) = v_{0}$. 
 
 Additionally, if $u_{0}$ and $v_{0} \in H^{s}(M)$, for $s > \frac{\sigma}{2}$, are small enough, we have that $h \in C([0, T]; H^{s}(M))$ and $u \in C([0, T]; H^{s}(M))$. 
\end{teo}
\begin{rem}
Note that assumption $(\mathcal{B})$ is not necessary to get the controllability result for small  data. However, our argument to prove  Theorem \ref{controle} combine Theorem \ref{teo-stab} with a fixed point argument, which justifies the necessity of assumption $(\mathcal{B})$ in Theorem \ref{controle}.
\end{rem}

\subsection{Heuristic and structure of the paper}
In this article, our goal is to give an answer for the two control problems mentioned at the beginning of this introduction. We adopt the approach in Dehamn-Gerard-Lebeau’s paper \cite{dehman-gerard-lebeau} to prove global control results for the fractional Schrödinger equation on compact manifolds without boundary. 
 Precisely, when we consider $\sigma=2$, we recover the result proved by Dehman \textit{et al.} \cite{dehman-gerard-lebeau} for Schr\"odinger equation. On the other hand, when we consider $\sigma=4$ and $d=1$, we can use the torus $\mathbb{T}$ instead of the manifold $M$ to get the result  proved by Capistrano-Filho and Cavalcante in \cite{CaCa}. 
 Thus, this work give us more general results towards the dimension of the manifolds and more general answers for the control problem for Schr\"odinger-like equations. Let us describe briefly the main arguments of the proof of the theorems presented in the previous subsection. 
 
 In the first result of the manuscript, we will show that assuming assumptions $(\mathcal{A})$ and  $(\mathcal{B})$ the system \eqref{prob-diss1} is asymptotically stable, that is, Theorem \ref{teo-stab} holds. To do this, we generalize the ideas introduced in \cite{dehman-gerard-lebeau} for the fractional Schr\"odinger operator, in other words, we apply microlocal analysis in a general context to achieve the result.

Remark that the local control result (Theorem \ref{main1}) is a consequence of the assumption $(\mathcal{A})$, it implies that the linear system associated to \eqref{prob-diss12} is controllable. In fact, the main novelty of the manuscript is that assumption ($\mathcal{A}$), the so-called \textit{Geometric Control Condition}, ensures such controllability conditions for the generalized nonlinear fractional Schr\"odinger equation in a compact Riemannian manifold.  Additionally, note that relationship between assumption $(\mathcal{A})$ and the controllability and stabilization problems is not an immediate consequence of \cite{dehman-gerard-lebeau} for the fractional Schr\"odinger model (for details see Section \ref{sec6}).  So, roughly speaking, with the linear control result in hand, using fixed point argument for initial and final data small enough, we can prove that assumption $(\mathcal{A})$ implies the local controllability for the nonlinear system \eqref{Ncontrolsys-int}. Here, it is important to point out that we do not use any result of unique continuation property (assumption $(\mathcal{B})$) to prove the local controllability results.

It is important to point out that in \cite{dehman-gerard-lebeau} the results are based on properties of the Laplace-Beltrami operator. Such an operator is well known in the literature and can be explicitly \linebreak characterized as well as  its principal symbol. However, in our case, we deal with the fractional Laplace-Beltrami operator. Note that, in order to study the behavior of this operator, two points are necessary, namely, the behavior of its eigenvalues (see Section \ref{sec2}) and most importantly, its principal symbol what is detailed in Remark \ref{remsym} below, which is one of the key point to prove the propagation of singularities results (Propositions \ref{reg-propagacao-prop1} and \ref{propagacaocompacidade}).

To finish, about the control result for large data (Theorem \ref{controle}), the proof is a combination of a global stabilization result (Theorem \ref{teo-stab}) and the local control result (Theorem \ref{main1}), as is usual in control theory, see e.g. \cite{Dehman,dehman-gerard-lebeau,Laurent-siam,Laurent-esaim}.  

\medskip

To end the introduction, we present the outline of our manuscript:

\medskip

-- Section \ref{sec2} is to establish preliminary results which were used throughout the paper, precisely, first, we collect the result of pseudo-differential calculus on manifolds. Additionally, we give the estimates needed in our analysis, namely, \textit{Strichartz estimates}. Finally, we prove the existence of a solution for the nonlinear fractional Schr\"odinger equation \eqref{prob-diss1} with source and damping terms.

\medskip

-- Next, Section \ref{sec3}, \textit{propagation of singularities} and \textit{unique continuation property} are proved and, with this in hand, Sections \ref{sec4} and \ref{sec5} are aimed to present the proof of the stabilization and controllability theorems, respectively. 

\medskip

-- Finally, we present in the Section \ref{sec6} concluding remarks and open problems.

\section{Preliminaries}\label{sec2}
In this section let us remember some results about pseudo-differential calculus on manifolds, more precisely, we are particularly interested  in giving more information about the following operators
$$\Lambda^{\sigma}_{g} = (\sqrt{-\Delta_{g}})^{\sigma} \text{ and } (1 - \Delta_{g})^{\frac{s}{2}},$$ for $\sigma\in[2,\infty)$ and $s \in \mathbb{R}$. In addition, we will give the well-posedness result for the system in consideration in this work.

It is important to note that, given a compact Riemannian manifold $M$ without boundary with metric $g$,  we shall denote by  $TM$ its tangent 
bundle, and by $T^{*}M$ its cotangent bundle. Then the co-sphere bundle is defined as:
$$
S^{*} M=\left\{(x, \eta) \in T^{*} M,|\eta|_{x}^{2}=1\right\}
$$
where $|\eta|_{x} = \sqrt{g_{x}(\eta,\eta)}$. 

\subsection{Pseudo-differential operator}
Let $P: C^{\infty}(M) \to C^{\infty}(M)$ be  a classical pseudo-differential operator of order 
$n \in \mathbb{N}$ with principal symbol $p := p(x, \xi) \in C^{\infty}(T^{*}M\backslash \{ 0 \})$. 
\begin{defin}
We say that $P$ is an elliptic pseudo-differential operator if $p$ does not vanish in 
$T^{*}M\backslash \{0 \}$.
\end{defin}

 Assume that  $n > 0$ and $P$ is an elliptic and self-adjoint pseudo-differential operator of order $n$. Also, let us assume that the principal symbol of $P$ is identically positive on $T^{*}M\backslash \{0\}$. As a consequence of the  spectral Theorem, there is an orthonormal basis $\{e_{j}\}$ to $L^{2}(M)$ of eigenvectors of $P$ associated to the eigenvalues $(\lambda_{j})$, such that 
\begin{eqnarray*}
P = \sum_{j=1}^{\infty}{\lambda_{j}E_{j}},
\end{eqnarray*}
where $E_{j}f = (f, e_{j})_{L^{2}(M)}e_{j}$, for $f \in L^{2}(M)$ and $i \in \mathbb{N}$. 

\begin{defin}
We say that $m \in C^{\infty}(\mathbb{R})$ is a symbol of order $\mu \in \mathbb{R}$ if  we have
\begin{eqnarray}\label{m-simb}
\left|\frac{d^{\alpha}}{d\lambda^{\alpha}}m(\lambda) \right| \leq C_{\alpha}(1 + |\lambda|)^{\mu - \alpha},
\end{eqnarray}
for all $\alpha \in \mathbb{N}$.
\end{defin}

Consider the operator $m(P): C^{\infty}(M) \to C^{\infty}(M)$ defined by
\begin{eqnarray}\label{def-mP}
m(P)u = \sum_{j=1}^{\infty}{m(\lambda_{j})E_{j}u}.
\end{eqnarray}
The following result holds true for $P$ and $m(P)$.

\begin{teo}\label{teo-apendice}
Let $P: C^{\infty}(M) \to C^{\infty}(M)$ an 
elliptic and self-adjoint pseudo-differential operator of order $n \in \mathbb{N}$, as above. Therefore, we have:
\begin{itemize}
   \item[$(i)$] The pseudo-differential operator $P^{\frac{1}{n}}$, defined by the spectral theorem, is a classical pseudo-differential operator with order $1$ and its  principal symbol is  given by $(p(x, \xi))^{\frac{1}{n}}$.
  
   \item[$(ii)$] If $P$ has order $1$, then operator $m(P): C^{\infty}(M) \to C^{\infty}(M)$, defined in \eqref{def-mP}, is a pseudo-differential operator of order $\mu$ with principal symbol $m(p(x, \xi))$.
\end{itemize}
\end{teo}
\begin{proof}
The proof can be found in \cite[Theorems 3.3.1 and 4.3.1]{Sogge} and, therefore, we will omit it.
\end{proof}

Let us now apply Theorem \ref{teo-apendice} for the negative Laplace-Beltrami operator, namely $P := -\Delta_{g}$ with principal symbol $p(x, \xi) = |\xi|_{x}^{2}$. Thus, as a consequence of item $(i)$ in Theorem \ref{teo-apendice}, $P_{1} = \sqrt{-\Delta_{g}}$ is a pseudo-differential operator of order $1$ with principal symbol 
$p_{1}(x, \xi) = |\xi|_{x}$ and $P_{\sigma} = \Lambda^{\sigma}_{g} = \left( \sqrt{-\Delta_{g}} \right)^{\sigma}$, for $\sigma\in[2,\infty)$, is a pseudo-differential operator of order $\sigma$ with principal symbol $p_{\sigma}(x, \xi) = |\xi|_{x}^{\sigma}$. Note that in the same way we can define the operators $(1 - \Delta_{g})^{\frac{s}{2}}$ as a pseudo-differential operator of order $s \in \mathbb{R}$.

\begin{rem}\label{remsym}
Let $(M, g)$ be a compact Riemannian manifold and $P_{1} = \sqrt{-\Delta_{g}}$ defined as above. Given $\omega_{0} \in T^{*}M\backslash \{0\}$, $\omega_{1} \in \Gamma_{\omega_{0}}$ and  $\Gamma_{\omega_{0}}$ geodesics starting at $\omega_{0}$ with speed $1$, consider $U$ and $V$ conic neighborhoods of $\omega_{0}$ and $\omega_{1}$, respectively. Thus, if  $c = c(x, \xi)$ is a symbol of order $s-1$, with $s \in \mathbb{R}$, supported in $U$, there exist $b, r$ symbols of order $s-1$ such that 
\begin{equation}
\frac{1}{i}\{ |\xi|_{x}, b\} = c + r,
\end{equation}
with $r$ supported in the conic neighborhood $V$ of $\omega_{1}$. From this, we can conclude that for any symbol $b_{1}$, since
\begin{equation}
\{ |\xi|^{\sigma}_{x}, b_{1}\} = \sigma|\xi|_{x}^{\sigma-1}\{|\xi|_{x}, b_{1}\},
\end{equation} 
for a given symbol $c = c(x, \xi)$ of order $s_{1} \in \mathbb{R}$ supported in $U$, there exists $b$, a symbol of order 
$s_{1} - \sigma + 1$, and $r$, a symbol of order $s_{1}$ and supported in $V$, such that
\begin{equation}
\frac{1}{i}\{|\xi|^{\sigma}_{x}, b_{1}\} = c + r.
\end{equation}
\end{rem}

\subsection{Cauchy problem for the fractional Schr\"odinger equation} 
 
In this section, the Strichartz estimates for the linear equation are revisited, these estimates play a crucial role in the whole work. Consider $\sigma \in [2, \infty)$ and $(M, g)$ a smooth compact Riemannian manifold without boundary. 
Here, as mentioned in the introduction, $\Delta_{g}$ is  the 
Laplace-Beltrami operator  associated to the metric $g$ and $\Lambda_{g}^{\sigma} = \left(\sqrt{-\Delta_{g}}\right)^{\sigma}$. The following result is borrowed from \cite{Burq-Gerard-Tz} and \cite{Dinh}.

\begin{proposition}[Strichartz estimates]\label{St-Ref}
In the conditions above, let $I \subset \mathbb{R}$ be 
a bounded interval and $(p,q)$ such that 
\begin{eqnarray}\label{St-Admssi}
p \in [2, \infty], \quad q \in [2, \infty) \quad (p, q, d) \neq (2, \infty, 2), \quad \frac{2}{p} + \frac{d}{q} = \frac{d}{2}
\end{eqnarray} 
then, there exists $C > 0$ such that 
\begin{eqnarray}\label{ref}
\|e^{it \Lambda_{g}^{\sigma}}u_{0}\|_{L^{p}(I; L^{q}(M))} \leq C\|u_{0}\|_{H^{\frac{1}{p}}(M)}.
\end{eqnarray}
Moreover, if $u=u(x,t)$ is a weak solution of 
\begin{equation*}
\begin{cases}
 i \partial_{t}u + \Lambda_{g}^{\sigma}u = F, &\text{ on } M\times I,\\
 u(x,0) = u_{0}(x), &x\in M,
\end{cases}
\end{equation*}
then,
\begin{eqnarray} \label{St-Est}
\|u\|_{L^{p}(I; L^{q}(M))} \leq C\left(\|u_{0}\|_{H^{\frac{1}{p}}(M)} + \|F\|_{L^{1}\left(I; H^{\frac{1}{p}}(M)\right)} \right).
\end{eqnarray}
\end{proposition}
\begin{rem}
The previous proposition is borrowed from  \cite{Dinh}, where it is stated with the norm in the right hand side of (\ref{ref}) is calculated in $H^{\gamma_{pq} - \frac{\sigma-1}{p}}(M^{d})$, with 
$$ \gamma_{p q}=\frac{d}{2}-\frac{d}{q}-\frac{\sigma}{p}.$$
In our case, since we have chosen $\frac{2}{p}+\frac{d}{q}=\frac{d}{2}$, this leads us to
\begin{eqnarray*}
\gamma_{pq} - \frac{\sigma-1}{p} &=& \frac{d}{2} - \frac{d}{q} - \frac{\sigma}{p} + \frac{\sigma}{p} - \frac{1}{p}\\
&=& \frac{d}{2} - \frac{d}{q} - \frac{1}{p} \\
&=& \frac{2}{p} - \frac{1}{p} = \frac{1}{p}.
\end{eqnarray*}
Therefore, Proposition \ref{St-Ref} can be seen as a particular case of \cite[Theorem 1.2]{Dinh}.
\end{rem}

With these Strichartz estimates in hand, we are able to infer that the Cauchy problem
\begin{equation} \label{prob-nonh}
\begin{cases}
i \partial_{t} u  + \Lambda^{\sigma}_{g}u  + P'(|u|^{2})u = h,&   \text{ on } M\times \mathbb{R}_{+}, \\
u(x,0)=u_{0}(x),& x\in M,
\end{cases}
\end{equation}
is globally well posed on compact Riemannian manifolds without boundary. More precisely, the result can be read as follows.
\begin{teo}\label{Teo-Prob-nh}
Let $\sigma \in [2, \infty)$ and $(M, g)$ be a compact smooth manifold of dimension $d$.
Then given $P$ a polynomial function with degree $d^{\circ}P\geq 1$, satisfying conditions $\eqref{p1}$ and $\eqref{p2}$, such that
\begin{eqnarray*}
s > \frac{d}{2} - \frac{1}{\max\{2d^{\circ}P -1, 2  \}}, \quad s\in\mathbb{R},
\end{eqnarray*}
for every $u_{0} \in H^{s}(M)$ and $h \in L^{1}_{loc}(\mathbb{R}; H^{s}(M))$ there exists $T > 0$ and a unique solution $u=u(x,t)$ of  \eqref{prob-nonh} with the following regularity $$u \in C([0, T]; H^{s}(M)).$$
 Moreover, if $d < [\sigma] + 1$ and  $s \geq \frac{\sigma}{2}$ the solution is global in time and $$u \in L^{p}(0, T; L^{\infty}(M)),$$ for every $p < \infty$ and for all $T < \infty$.
\end{teo}

\begin{proof}
	Let $T > 0$ and pick $p > \max\{\beta - 1, 2 \}$, where $\beta = 2d^{\circ}P$,
such that $s > \frac{d}{2} - \frac{1}{p}$ and consider 
\begin{eqnarray}\label{def-Yt}
Y_{T} = C([0, T]; H^{s}(M)) \cap L^{p}([0, T]; W^{\alpha, q}(M)),
\end{eqnarray}
where $q$ is given by $\frac{2}{p} + \frac{d}{q} = \frac{d}{2}$ and $\alpha = s - \frac{1}{p} > \frac{d}{q}$. Consider in $Y_T$ the following norm
\begin{eqnarray*}
\|u\|_{Y_{T}} = \max_{0 \leq t \leq T}\|u(t)\|_{H^{s}(M)} + \|(1 - \Delta_{g})^{\frac{\alpha}{2}}u\|_{L^{p}([0, T]; L^{q}(M))}.
\end{eqnarray*}

Observe that $Y_{T} \subset L^{p}([0, T]; L^{\infty}(M))$ by Sobolev embeddings. Due to the Duhamel formula, we have to prove that the operator 
\begin{eqnarray*}
\Phi(u)(t) = e^{it \Lambda_{g}^{\sigma}}u_{0} - i\int_{0}^{t}{e^{i(t-\tau)\Lambda_{g}^{\sigma}}[h - P'(|u|^{2})u](\tau)}d\tau,
\end{eqnarray*}
has a fixed point. In fact,  by using Strichartz estimates given in \eqref{St-Est}, we get that
\begin{eqnarray}\nonumber
\|\Phi(u)\|_{Y_{T}} & \leq & C\left(\|u_{0}\|_{H^{s}(M)} + \int_{0}^{T}{\|h(\tau) - P'(|u|^{2})u(\tau)\|_{H^{s}(M)}} d\tau\right)\\   \label{contd-1}
& \leq & C\left(\|u_{0}\|_{H^{s}(M)} + \|h\|_{L^{1}([0, T]; H^{s}(M))} + \int_{0}^{T}{(1 + \|u\|_{L^{\infty}(M)}^{\beta-1})\|u(\tau)\|_{H^{s}(M)}}d\tau \right) \\
& \leq &C\left(\|u_{0}\|_{H^{s}(M)} + \|h\|_{L^{1}([0, T]; H^{s}(M))} + T^{\gamma}(1 + \|u\|_{Y_{T}}^{\beta-1})\|u\|_{L^{\infty}([0, T]; H^{s}(M))}\right),\nonumber
\end{eqnarray}
where $\gamma = 1 - \frac{\beta-1}{p}$. Similarly, for every $u, v \in Y_{T}$,
\begin{eqnarray}\label{cont-2-ex}
\|\Phi(u) - \Phi(v)\|_{Y_{T}} \leq CT^{\gamma}(1 + \|u\|_{Y_{T}}^{\beta - 1} + \|v\|_{Y_{T}}^{\beta - 1})\|u - v\|_{Y_{T}}.
\end{eqnarray}
To conclude, consider $\Phi$ defined in a closed ball $$B_{R} = \{u \in Y_{T}; \|u\|_{Y_{T}} \leq R \},$$ with $R, T > 0$ small enough. Thanks to \eqref{contd-1} and \eqref{cont-2-ex}, $\Phi$ is a contraction and, thus, has a unique fixed point, i.e., \eqref{prob-nonh} has a local solution defined in a maximal interval $[0, T]$.  Arguing as in \eqref{cont-2-ex}, if $r \geq \frac{\sigma}{2}$, $u_{0} \in H^{r}(M)$ and $h \in L^{1}_{loc}(0, T; H^{r}(M))$, we get that
\begin{eqnarray*}
\|\Phi(u) - \Phi(0)\|_{L^{\infty}(0, T; H^{r}(M))} \leq CT^{\gamma}\left(1 + \|u\|_{Y_{T}} \right) \|u\|_{L^{\infty}(0, T; H^{r}(M))},
\end{eqnarray*}
which ensures that \eqref{prob-nonh} has a local solution $u \in L^{\infty}(0, T; H^{r}(M))$.

Next, let us prove that if $s \geq \frac{\sigma}{2}$, $u=u(x,t)$ solution of \eqref{prob-nonh} is global in time. To see this, we consider the energy functional $E(t)$ defined by  \eqref{Func-En}. Thus,  for every $T > 0$, observe that
 \begin{eqnarray*}
 E(T) \leq E(0) + \int_{0}^{T}{\int_{M}{|\Lambda_{g}^{\frac{\sigma}{2}}u| |\Lambda_{g}^{\frac{\sigma}{2}}h|}dx}dt + \int_{0}^{T}{\int_{M}{|P'(|u|^{2})u| |h|}dx}dt.
 \end{eqnarray*}
Let us to bound the last integral. Denoting $\alpha_{1} = \frac{\beta'}{\beta'-1}$, where $\beta' = 2d^{\circ}P$ and $\frac{1}{\alpha_{1}} + \frac{1}{\mu} = 1$, we estimate the last integral as follows
\begin{eqnarray*}
\int_{0}^{T}{\int_{M}{|P'(|u|^{2})u| |h|}dx}dt & \leq & \int_{0}^{T}{\left[\int_{M}{|P'(|u|^{2})u(t)|^{\alpha_{1}}}dx \right]^{\frac{1}{\alpha_{1}}} \|h(t) \|_{L^{\mu}(M)}}dt \\
& \leq & C \int_{0}^{T}{\left[1 + \int_{M}{|u(t)|^{\beta'}}dx \right] \|h(t)\|_{H^{\frac{\sigma}{2}}(M)}}dt \\
& \leq & C\int_{0}^{T}{\left[ 1+ \int_{M}{P(|u|^{2})}dx \right]\|h(t) \|_{H^{\frac{\sigma}{2}}(M)}}dt.
\end{eqnarray*}
Thus, we get the following
\begin{eqnarray*}
E(T) \leq E(0) + C\int_{0}^{T}{(1 + E(t))\|h(t) \|_{H^{\frac{\sigma}{2}}(M)}}dt.
\end{eqnarray*}
An application of Gronwall inequality give us that $E(t)$ is uniformly bounded, hence $u=u(x,t)$, solution of \eqref{prob-nonh},  is global in time.  Moreover, since $u \in Y_{T}$, for  every $T > 0$, we have that $u \in \bigcap_{p< \infty} L^{p}_{loc}(\mathbb{R}; L^{\infty}(M))$. 

To finish our proof, it remains to show that the solution is unique in $C([0, T]; H^{s}(M))$, for $s \geq \frac{\sigma}{2}$. Consider  $u$ and $v$ solutions of \eqref{prob-nonh}. Therefore, we have that $u-v$ satisfies 
\begin{equation}\label{eq-uni-1}
\begin{cases}
i \partial_{t}(u - v)  + \Lambda_{g}^{\sigma}(u-v)= P'(|v|^{2})v - P'(|u|^{2})u,&(x,t)\in  M\times ]0, T[,\\
(u - v)(0)= 0,&x\in M.
\end{cases}
\end{equation}
By the hypothesis over $d$, that is, $d < [\sigma] + 1$, we have that $H^{\frac{\sigma}{2}}(M) \hookrightarrow L^{q}(M)$ for every $q \in [2, \infty)$, if $d = \sigma$ and   $H^{\frac{\sigma}{2}}(M) \hookrightarrow L^{\infty}(M)$, if $d < [\sigma]$. From \eqref{eq-uni-1} follows that
\begin{eqnarray*}
\|u-v\|_{L^{\infty}(0, T; L^{2}(M))} & \leq & C\int_{0}^{T}{\int_{M}{\left|P'(|u|^{2})u - P'(|v|^{2})v\right| \left|u - v \right|}}dt \\
& \leq & C(u,v)T\|u - v\|_{L^{\infty}(0, T; L^{2}(M))},
\end{eqnarray*}
where $C(u, v) > 0$ denotes a constant depending on $u$ and $v$. Hence, for $T' > 0$ small enough, we conclude that $u = v$ on $]0, T'[ \times M$, iterating this result we get that $u = v$ on $]0, T[ \times M$. 
\end{proof}

\begin{remarks}\label{Remarks}
The following remarks are now in order.
\begin{itemize}
\item[i.] Observe that, by using estimate \eqref{contd-1} with $T < 1$ small, if the constant $C$ is big enough yields that 
\begin{eqnarray*}
\|u\|_{Y_{T}} \leq C\left(\|u_{0}\|_{H^{s}(M)} + \|h\|_{L^{1}(0, T; H^{s}(M))} + \|u\|_{Y_{T}}^{\beta} \right).
\end{eqnarray*}  
Therefore, if $\|u_{0}\|_{H^{s}(M)} + \|h\|_{L^{1}(0, T; H^{s}(M))}$ is small enough, we can conclude, by a boot-strap argument, that 
\begin{eqnarray*}
\|u\|_{Y_{T}} \leq C(\|u_{0}\|_{H^{s}(M)} + \|h\|_{L^{1}(0, T; H^{s}(M))}).
\end{eqnarray*}
\item[ii.] From the proof of Theorem \ref{Teo-Prob-nh} we can conclude that the solution $u\in C(\mathbb{R}_+; H^{\frac{\sigma}{2}}(M))$ of \eqref{prob-nonh}  is such that $u \in L^{p}(0, T; L^{\infty}(M))$, for all $T > 0$. Hence, we have that $P'(|u|^{2})u \in L^{2}(0, T; H^{\frac{\sigma}{2}}(M))$, for all $T > 0$ and 
arguing as in \eqref{contd-1} we have, if $h = 0$,  the following estimate
\begin{eqnarray*}
\|P'(|u|^{2})u\|_{L^{2}(0, T; H^{\frac{\sigma}{2}}(M))} \leq C(T)\left(\|u\|_{Y_{T}}^{2} + \|u\|_{Y_{T}}^{2 \beta} \right),
\end{eqnarray*}
for all $T> 0$.
\end{itemize}
\end{remarks}





\subsection{Well-posedness for the full system}Finally, to finish this section, let us prove a result that ensure the existence of solutions  for the nonlinear fractional Schr\"odinger equation with damping term, that is, changing $h$ by $a(x)(1 - \Delta_{g})^{-\frac{\sigma}{2}}a(x)\partial_{t}u$ in the system \eqref{prob-nonh}, the result is the following one.

\begin{teo}\label{Prob-diss-Teo}
Let $\sigma \in [2, \infty)$ and $(M, g)$ be a compact Riemannian manifold of dimension $d < [\sigma] +1$. Then given $P'$ a polynomial function satisfying conditions \eqref{p1}- \eqref{p2},  
$u_{0} \in H^{\frac{\sigma}{2}}(M)$ and 
$a=a(x) \in C^{\infty}(M)$ a non-negative  real valued function,
then there exists an unique $u \in C(\mathbb{R}_{+}; H^{\frac{\sigma}{2}}(M))$ solution of the system
  \begin{equation} \label{prob-diss}
  \begin{cases}
i \partial_{t} u  + \Lambda^{\sigma}_{g}u  + P'(|u|^{2})u - a(x)(1 - \Delta_{g})^{-\frac{\sigma}{2}}a(x)\partial_{t}u = 0, &\text{ on } M\times \mathbb{R}_{+} ,\\
u(x,0)=u_0(x),& x\in M.
\end{cases}
\end{equation}

\end{teo}
\begin{proof} 
We claim that:

\vspace{0.2cm}
\textit{The operator $Jv = (1 - ia(x)(1 - \Delta_{g})^{-\frac{\sigma}{2}}a(x))v$ is a pseudo-differential operator of order $0$ which defines an isomorphism on $H^{s}(M)$, for $s \in \mathbb{R}$, and also on $L^{p}(M)$.}
\vspace{0.2cm}

Indeed, note that we can write $J$ as $J = I + J_{1}$, where $J_{1}$ is an anti-self-adjoint operator in $L^{2}(M)$ (we refer the reader to the Section \ref{sec2} for an introduction of such kind of operators), thus $J$ is an isomorphism in $L^{2}(M)$ and, due to the ellipticity, in $H^{s}(M)$, for $s>0$. Note that  $J$ is an isomorphism for every $s \in \mathbb{R}$, by duality. 

With this information in hand, system \eqref{prob-diss} can be written as follows
  \begin{equation*}
  \begin{cases}
\partial_{t}v - i\Lambda^{\sigma}_{g}v - R_{0}v -iP'(|u|^{2})u = 0, &\text{ on } M\times ] 0, T[ ,\\
v = Ju, &\text{ on } M\times ] 0, T[ ,\\
v(0) = v_{0} = Ju_{0} \in H^{\frac{\sigma}{2}}(M), & \text{ on } M,
\end{cases}
\end{equation*}
where $R_{0} = - i\Lambda_{g}^{\sigma} + i \Lambda_{g}^{\sigma}J^{-1}$ is a pseudo-differential operator of order $0$. Observe that, by the Duhamel formula, as in the previous proof, the functional
  \begin{eqnarray}
  \Phi(v)(t) = e^{it\Lambda_{g}^{\sigma}}v_{0} + \int_{0}^{t}{e^{i(t - \tau)\Lambda_{g}^{\sigma}}[R_{0}v + iP'(|u|^{2})u](\tau)}d\tau
  \end{eqnarray}
  has a fixed point, considering $\Phi$ defined in a suitable ball of the space $Y_{T}$ defined in  \eqref{def-Yt}, with $s = \frac{\sigma}{2}$, which provides the local existence. Considering the functional $E(t)$ defined in \eqref{Func-En} we have that 
$$
  E(t) \leq E(0), \forall t \in [0, T],
$$
  which guarantees that this solution is global in time. The uniqueness can be proved as in Theorem \ref{Teo-Prob-nh}, and the proof is complete.
\end{proof}

\section{Propagation of singularities and unique continuation property}\label{sec3}

\subsection{Propagation of regularity} We begin this section by proving a result of propagation of regularity for solutions to the linear fractional Schr\"odinger equation, which were used throughout the paper. The main ingredient is basic pseudo-differential analysis.
\begin{proposition}\label{reg-propagacao-prop1}
Let $u\in C([0, T]; H^{\frac{\sigma}{2}}(M))$ be a solution of 
$$
i\partial_{t}u + \Lambda_{g}^{\sigma}u = f,
$$
with $f \in L^{2}_{loc}(]0, T[; H^{\frac{\sigma}{2}}(M))$ and consider 
$\omega \subset M$ an open set satisfying assumption $(\mathcal{A})$.  If $u \in L^{2}(]0, T[; H^{\frac{\sigma}{2} + \rho}(\omega))$ for some 
$\rho \leq \frac{1}{2}$ 
then $$u \in L^{2}(]0, T[; H^{\frac{\sigma}{2} + \rho}(M)).$$ In particular, 
if $u \in C^{\infty}(]0, T[ \times \omega)$, then  $u \in C^{\infty}(]0, T[ \times M)$.
\end{proposition}

\begin{proof}
Observe that for every $\omega_{0} = (x_{0}, \xi_{0}) \in T^{*}M\backslash \{0\}$, with $x_{0} \in \omega$, 
we can consider $\phi(x, D_{x})$ a zero order pseudo-differential operator elliptic at $\omega_{0}$ 
and, such that, $$\phi(x, D_{x})u \in L^{2}_{loc}(0, T; H^{\frac{\sigma}{2} + \rho}(M)).$$
The result will be proved by using elliptic regularity and assumption $(\mathcal{A})$, showing that for every $\omega_{1} \in \Gamma_{\omega_{0}}$, $\Gamma_{\omega_{0}}$ a geodesics of $M$ starting at $\omega_{0}$ and traveling with speed $1$, one has that there exists $\psi(x, D_{x})$ a pseudo-differential operator of order $0$, elliptic at $\omega_{1}$, such that 
$\psi(x, D_{x})u \in L^{2}_{loc}(0, T; H^{\frac{\sigma}{2}+\rho}(M))$. 

First, consider $B(x, D_{x})$ a tangential pseudo-differential 
operator on $M$ of order $2s - (\sigma - 1)$, where 
$s = \frac{\sigma}{2}+\rho$ and $\varphi \in C_{0}^{\infty}(]0, T[)$. 
Then, $A(t, x, D_{x}) = \varphi B(x, D_{x})$ is a tangential pseudo-differential operator of order 
$2s - (\sigma -1)$. Now, regularize $u$ by introducing the sequence 
$$u_{n} = \left(1 - \frac{1}{n}\Delta_{g}\right)^{-2}u$$
and denote $L = i \partial_{t} + \Lambda_{g}^{\sigma}$, thus we have 
\begin{eqnarray*}
(Lu_{n}, Au_{n})_{L^{2}(0, T; L^{2}(M))} - (Au_{n}, Lu_{n})_{L^{2}(0, T; L^{2}(M))} =& ([A, \Lambda_{g}^{\sigma}]u_{n}, u_{n})_{L^{2}(0, T; L^{2}(M))}\\& - i(\varphi'Bu_{n}, u_{n})_{L^{2}(0, T; L^{2}(M))},
\end{eqnarray*}
which leads to
\begin{eqnarray}\label{prop-reg1}
\begin{split}
([A, \Lambda_{g}^{\sigma}]u_{n}, u_{n})_{L^{2}(0, T; L^{2}(M))} &&= (Lu_{n}, Au_{n})_{L^{2}(0, T; L^{2}(M))} - (Au_{n}, Lu_{n})_{L^{2}(0, T; L^{2}(M))} \\&&+ i(\varphi'Bu_{n}, u_{n})_{L^{2}(0, T; L^{2}(M))}.
\end{split}
\end{eqnarray}
We need to bound the terms that appears in \eqref{prop-reg1}, which will give us that 
\begin{eqnarray}
([A, \Lambda_{g}^{\sigma}]u_{n}, u_{n})_{L^{2}(0, T; L^{2}(M))}  = \int_{0}^{T}{\varphi(t)([B, \Lambda_{g}^{\sigma}]u_{n}, u_{n})_{L^{2}(M)}}dt
\end{eqnarray}
is uniformly bounded with respect to $n \in \mathbb{N}$. Indeed, to estimate \eqref{prop-reg1} we 
observe that 
\begin{eqnarray*}
\left|(Lu_{n}, Au_{n})\right| &\leq& \|Lu_{n}\|_{L^{1}\left(0,T; H^{-s+\frac{1}{2}}(M)\right)}\|Au_{n}\|_{L^{\infty}\left(0, T; H^{s-\frac{1}{2}}(M)\right)} \\
& \leq & C\|u_{n}\|_{L^{\infty}(0, T;H^{r}(M))}\|Lu_{n}\|_{L^{1}(0, T; H^{r}(M))},
\end{eqnarray*}
where $r = s - \sigma - \frac{1}{2}$ and  we have used that $$A: L^{\infty}(0, T; H^{s - \sigma - \frac{1}{2}}(M)) \to L^{\infty}(0, T; H^{-s+\frac{1}{2}}(M)),$$ is a continuous linear operator and 
$\rho \leq \sigma + \frac{1}{2}$. The other terms in $(\ref{prop-reg1})$ can be analogously estimated, then we will omit the details.

Now, take $\omega_{1} \in \Gamma_{\omega_{0}}$, 
$U$ and $V$ two small conic neighborhoods of $\omega_{1}$ and $\omega_{0}$, respectively. As observed in Remark \ref{remsym},  
for every symbol $c = c(x, \xi)$ of order $s$ and  supported 
in $U$ there exists a symbol $b(x, \xi)$ of order 
$2s - \sigma +1$ and a symbol $r(x, \xi)$ of order $2s$ supported in $V$ such that 
\begin{eqnarray*}
\frac{1}{i}\{|\xi|^{\sigma}_{x}, b \} = |c(x, \xi)|^{2} + r(x, \xi) .
\end{eqnarray*}
If we choose $c$ elliptic in $\omega_{1}$ then we conclude that 
\begin{eqnarray*}
\int_{0}^{T}{\|c(x, D_{x})u_{n}(t)\|_{L^{2}(M)}^{2}}dt \leq C.
\end{eqnarray*}
Lastly, defining $\psi(x, D_{x}) = (1 - \Delta_{g})^{- \frac{s}{2}}c(x, D_{x})$ we get the desired result.

If $u \in C^{\infty}((0, T) \times \omega)$ we can repeat the same reasoning to prove that $u \in L^{2}(0,T; H^{\frac{\sigma}{2}+ 2 \rho}(M))$ and then to prove that $u \in L^{2}(0,T; H^{\frac{\sigma}{2}+ 3 \rho}(M))$ and, iterating this process, conclude that $u \in C^{\infty}((0, T) \times M)$.
\end{proof}

\subsection{Unique continuation property} We come now to prove that the solutions of the fractional Schr\"odinger equation when start smooth in a sub-domain $\omega$ of $M$ keep smooth in $M$, the result is read as follows.

\begin{cor}\label{prop-reg-proposition2}
Assume that $\omega \subset M$ is an open set satisfying assumption $(\mathcal{A})$. Consider $u \in C(]0, T[; H^{\frac{\sigma}{2}}(M))$ solution of 
\begin{equation}\label{prop-reg-2}
\begin{cases}
i \partial_{t} u+\Lambda^{\sigma}_{g}u +P'(|u|^{2})u=0,&\text{ on } M\times ] 0, T[,
\\ u(x,0)=u_0(x),& x\in M,
\end{cases}
\end{equation}
such that $u \in L^{2}(0, T; H^{\frac{\sigma}{2} + \rho}(\omega))$, for some $\rho > 0$. Then, $u \in C(]0, T[; H^{\frac{\sigma}{2}+\rho}(M))$. In particular, if $u \in C^{\infty}(]0, T[ \times \omega)$, then $u \in C^{\infty}(]0, T[ \times M)$.
\end{cor}
\begin{proof}
The result holds if we show that $$u \in L^{2}_{loc}(0, T; H^{\frac{\sigma}{2} + \rho}(M)).$$ Indeed,  choosing $t_{0} \in ]0, T[$ such that $u(\cdot,t_{0}) \in H^{\frac{\sigma}{2} + \rho}(M) $ and solving \eqref{prop-reg-2} with this initial data, by uniqueness of  solution, we have that $u \in C([0, T]; H^{\frac{\sigma}{2} + \rho}(M))$. As observed on the item ii. of the Remark \ref{Remarks}, $P'(|u|^{2})u \in L^{2}(0, T; H^{\frac{\sigma}{2}}(M))$. Now, we are in the conditions of the Proposition \ref{reg-propagacao-prop1}, then we conclude that $u \in L^{2}_{loc}(0, T; H^{\frac{\sigma}{2} + \rho}(M))$ as desired, and thus, the proof of the corollary is complete.
\end{proof}

We finish this section with a unique continuation property for the fractional Schr\"odinger equation.

\begin{cor}\label{teoUCP}
Let $\omega \subset M$ be an open set 
satisfying assumptions $(\mathcal{A})$ and $(\mathcal{B})$. 
Consider $$u \in C([0, T]; H^{\frac{\sigma}{2}}(M))$$
solution of 
\begin{eqnarray}\label{eq-prop-reg}
i\partial_{t}u + \Lambda_{g}^{\sigma}u + P'(|u|^{2})u = 0
\end{eqnarray}
such that $\partial_{t}u = 0$ on $]0, T[ \times \omega$, then $u = 0$ on 
$]0, T[ \times M$. 
\end{cor}

\begin{proof}
Observe that $u$ satisfies the equation 
\begin{eqnarray*}
\Lambda^{\sigma}_{g}u + P'(|u|^{2})u = 0, \quad \text{on } ]0, T[ \times \omega. 
\end{eqnarray*} 
Since $\Lambda^{\sigma}_{g}$ is an elliptic pseudo-differential operator we have 
$u \in C^{\infty}(]0, T[ \times \omega)$. Hence, by using
Proposition \ref{prop-reg-proposition2},
follows that $u \in C^{\infty}(]0, T[ \times M)$. Thus, taking the time derivative
of \eqref{eq-prop-reg} and defining  $v:= \partial_{t}u$, we see that $v$ satisfies the following system
\begin{equation*}
\begin{cases}
 i\partial_{t}v + \Lambda_{g}^{\sigma}v + b_{1}(t,x)v + b_{2}(t, x)\overline{v} = 0, &(x,t)\in M\times ] 0, T[,
\\ v = 0,& (x,t)\in \omega \times ] 0, T[,
\end{cases}
\end{equation*}
with $b_{1}, b_{2} \in C^{\infty}(]0, T[ \times M)$. Therefore, by assumption $(\mathcal{B})$, we 
get $v = \partial_{t}u = 0$ on $]0, T[ \times M$. We conclude the proof multiplying 
equation \eqref{eq-prop-reg} by $\overline{u}$ and 
integrating over $M$, then
\begin{eqnarray*}
\int_{M}{|\Lambda_{g}^{\frac{\sigma}{2}} u|^{2}}dx + \int_{M}{P'(|u|^{2})|u|^{2}}dx = 0,
\end{eqnarray*}
which yields $u \equiv 0$, thanks to \eqref{p2}. This finishes the proof of Corollary \ref{teoUCP}.
\end{proof}

\subsection{Propagation of compactness}
The final part of this section is related to show how the result of propagation of compactness can be obtained from the construction of suitable microlocal defect measures. It is important to point out that the construction of the tangential microlocal defect measure is classical, in this way we infer from the reader the article \cite{Gerard} for an interesting overview of this topic. 

In what follows, let $T>0$ and consider $\{u_{n}\}:=\{u_{n}\}_{n\in\mathbb{N}}$ 
be a sequence in $C([0, T]; H^{\frac{\sigma}{2}}(M))$ such that
\begin{eqnarray} \label{seq-prop-comp1}
\sup_{t \in [0, T]}\|u_{n}(t)\|_{H^{\frac{\sigma}{2}}(M)} \leq C, \\\label{seq-prop-comp2}
\sup_{t \in [0, T]}\|u_{n}(t)\|_{L^{2}(M)} \to 0
\end{eqnarray}
and
\begin{eqnarray}
\int_{0}^{T}{\|Lu_{n}(t)\|_{H^{\frac{\sigma}{2}}(M)}^{2}}dt \to 0. \label{seq-prop-comp3}
\end{eqnarray}
Here, the operator $L$ is defined by $L= i \partial_{t} + \Lambda^{\sigma}_{g}$, with $\Lambda^{\sigma}_{g}$ defined as before. Thus, we are able to present our first result of propagation of compactness.

\begin{proposition}\label{propagacaocompacidade}
Let $\{u_{n}\}$ be a sequence satisfying \eqref{seq-prop-comp1}--\eqref{seq-prop-comp3}, and assume that 
\begin{eqnarray*}
u_{n} \to 0 \text{ in } L^{2}([0, T]; H^{\frac{\sigma}{2}}(\omega)),
\end{eqnarray*} 
with $\omega \subset M$ an open set satisfying assumption $(\mathcal{A})$. Then, there exists a subsequence of $\{u_{n}\}$, still denoted by the same index, such that $$u_{n} \to 0 \text{ strongly in }L^{\infty}(0, T; H^{\frac{\sigma}{2}}(M)).$$
\end{proposition}
\begin{proof}
First, denote $\{v_{n}\}_{n \in \mathbb{N}}:=\{v_{n}\} = \{(1 - \Delta_{g})^{\frac{\sigma}{4}}u_{n}\}_{n\in\mathbb{N}}$. Thus, conditions \eqref{seq-prop-comp1}--\eqref{seq-prop-comp3} can be rewritten as 
 \begin{eqnarray} \label{seq-prop-comp4}
\sup_{t \in [0, T]}\|v_{n}(t)\|_{L^{2}(M)} \leq C, \\ \label{seq-prop-comp5}
\sup_{t \in [0, T]}\|v_{n}(t)\|_{H^{-\frac{\sigma}{2}}(M)} \to 0
\end{eqnarray}
and
 \begin{eqnarray} 
\int_{0}^{T}{\|Lv_{n}(t)\|_{L^{2}(M)}^{2}}dt \to 0. \label{seq-prop-comp6}
\end{eqnarray}
Since $$u_{n} \to 0 \text{ in } L^{2}(0, T; H^{\frac{\sigma}{2}}(\omega)),$$ 
we have that $$v_{n} \to 0 \text{ in }L^{2}(0, T; L^{2}(\omega)).$$
Therefore,  the result will be proved if we show that we can extract a subsequence of $\{v_{n}\}$, still denoted by the same index, such that $$v_{n} \to 0 \text{ in }L^{\infty}(0, T; L^{2}(M)).$$

To prove it, we split the proof into 3 steps. First, let us construct the microlocal defect measures.

\vspace{0.2cm}

\noindent\textbf{Step 1. Construction of the microlocal defect measure}

\vspace{0.2cm}
Under the assumptions \eqref{seq-prop-comp4} and \eqref{seq-prop-comp5}, by using the ideas contained in \cite{Gerard}, there exists  a 
tangential radon measure $\mu = \mu(t, x, \xi)$ such that 
\begin{eqnarray}\label{mdm}
(A(t, x, D_{x})v_{n}, v_{n})_{L^{2}(]0, T[ \times M)} \to \int_{[0, T] \times S^{*}M}{a(t, x, \xi)}d\mu(t, x, \xi),
\end{eqnarray}
for all zero-order tangential pseudo-differential operator $A = A(t, x, D_{x})$.

\vspace{0.2cm}

In the second step, we will prove that the microlocal defect measure describes precisely the information carried along the geodesics of $M$.

\vspace{0.2cm}

\noindent\textbf{Step 2. Propagation along the geodesics} 

\vspace{0.2cm}

Consider $\varphi \in C_{0}^{\infty}(]0, T[)$ and $B(x, D_{x})$ a pseudo-differential operator of order  $1 - \sigma$ with principal  symbol $b_{1-\sigma}$. Define, $$A(t,x, D_{x}) = \varphi B(x, D_{x})$$ and for 
 $\epsilon > 0$, $$A_{\epsilon} := \varphi B_{\epsilon} = Ae^{\epsilon \Delta_{g}}.$$ Moreover, denote
\begin{eqnarray*}
\alpha_{n}^{\epsilon} & = & (Lv_{n}, A_{\epsilon}^{*}v_{n})_{L^{2}} - (A_{\epsilon}v_{n}, Lv_{n})_{L^{2}}.
\end{eqnarray*} 

Note that, 
\begin{eqnarray*}
\left|(Lv_{n}, A_{\epsilon}^{*}v_{n}) \right| & \leq & \|Lv_{n}\|_{L^{2}L^{2}}\|A_{\epsilon}^{*}v_{n}\|_{L^{2}L^{2}} \\
& \leq & C\|Lv_{n}\|_{L^{2}L^{2}}\|v_{n}\|_{L^{2}H^{1-\sigma}}  \to 0.
\end{eqnarray*}
Analogously, we have  that
\begin{eqnarray*}
\left|(A_{\epsilon}v_{n}, Lv_{n}) \right| \to 0,
\end{eqnarray*}
and so $$\sup_{\epsilon > 0}\alpha_{n}^{\epsilon} \to 0.$$

On the other hand,
\begin{eqnarray*}
\alpha_{n}^{\epsilon} := ([A_{\epsilon}, \Lambda^{\sigma}_{g}]v_{n}, v_{n})_{L^{2}} - i ((\partial_{t}A_{\epsilon})v_{n}, v_{n})_{L^{2}}
\end{eqnarray*}
Note that,
\begin{eqnarray*}
\left|((\partial_{t}A_{\epsilon})v_{n}, v_{n})_{L^{2}} \right| & \leq & \|(\partial_{t}A_{\epsilon})v_{n}\|_{L^{2}H^{\frac{\sigma}{2}}}\|v_{n}\|_{L^{2}H^{- \frac{\sigma}{2}}} \\
& \leq & C\|v_{n}\|_{L^{2}H^{ - \frac{\sigma}{2} + 1}}\|v_{n}\|_{L^{2}H^{- \frac{\sigma}{2}}} \to 0.
\end{eqnarray*}
Therefore, this yields that
\begin{eqnarray*}
\sup_{\epsilon > 0}\left|((\partial_{t}A_{\epsilon})v_{n}, v_{n})_{L^{2}} \right|  \to 0.
\end{eqnarray*}
So, first taking $\epsilon \to 0$ and, after,  $n \to \infty$, we get that
 \begin{eqnarray*}
 (\varphi [B, \Lambda^{\sigma}_{g}]v_{n}, v_{n})_{L^{2}} \to 0
 \end{eqnarray*}
which means that,
\begin{eqnarray}\label{propg}
\int_{[0, T] \times S^{*}M}{\varphi(t)\{b_{1 - \sigma}, |\xi|_{x}^{\sigma} \}}d\mu = 0.
\end{eqnarray}

\vspace{0.2cm}

\noindent\textbf{Claim:} If $G_s$  denotes the geodesic flow on $S^{*}M =\left\{(x, \eta) \in T^{*} M,|\eta|_{x}^{2}=1\right\}$, condition \eqref{propg} implies that $\mu$ is invariant under the geodesic flow of $S^{*}M$, or equivalently, 
\begin{equation}\label{G_s}
G_s(\mu) = \mu.
\end{equation}

\vspace{0.2cm}

Indeed, first note that,
\begin{eqnarray*}
\{ b_{1 - \sigma}, |\xi|^{\sigma}_{x}\} = \left(\frac{\sigma}{2} \right)\{ b_{1 - \sigma}, |\xi|^{2}_{x}\}, 
\end{eqnarray*} 
on $S^{*}M$. Hence, by using \eqref{propg}, we conclude that,
\begin{eqnarray} \label{inv-metr}
\int_{[0, T] \times S^{*}M}{\varphi(t)\{b_{1-\sigma}, |\xi|^{2}_{x} \}}d\mu = 0.
\end{eqnarray}
Denote $H_{|\xi|^{2}_{x}}$ as the Hamiltonian flow associated with $p(x, \xi) = |\xi|^{2}_{x}$ and $$\Phi_{s}: (x_{0}, \xi_{0}) \mapsto (x_{s}, \xi_{s}), \quad s \in \mathbb{R},$$ the integral curves of $H_{|\xi|^{2}_{x}}$. Thus, we have that 
\begin{eqnarray*}
\frac{d}{ds}\int_{[0, T] \times S^{*}M}{\varphi(t)\left(b_{1-\sigma}\circ \Phi_{s}\right)}d\mu = \int_{[0, T] \times S^{*}M}{\varphi(t)\{b_{1-\sigma}, |\xi|^{2}_{x} \}}d\mu = 0.
\end{eqnarray*}
This identity precisely expresses property \eqref{G_s}  which guarantees that the claim holds.

\vspace{0.2cm}

Finally, on the third step, we prove the convergence of $\{v_n\}$, which implies the convergence of $\{u_n\}$ in all manifolds $M$, showing thus Proposition \ref{propagacaocompacidade}.

\vspace{0.2cm}

\noindent\textbf{Step 3. Convergence in $L^{\infty}(0, T; L^{2}(M))$}

\vspace{0.2cm}

From \eqref{mdm}, for all $f \in C^{\infty}(M)$, we have that
\begin{eqnarray*}
(fv_{n}, v_{n})_{L^{2}(]0, T[ \times M)} \to \int_{[0, T] \times S^{*}M}{f(x)}d\mu.
\end{eqnarray*}
Then, $$v_{n} \to 0 \text{ strongly in } L^{2}_{loc}(0, T; L^{2}(M))$$ if and only if $$\mu = 0, \quad \text{ on }S^{*}M\times[0, T].$$ In particular, by hypothesis of Proposition \ref{propagacaocompacidade}, we get $$v_{n} \to 0 \quad \text{ on } [0, T] \times \omega,$$ with $\omega \subset M$ an open set satisfying the assumption $(\mathcal{A})$. Therefore, holds that $\mu = 0$ on $[0, T] \times S^{*}M$, hence $$v_{n} \to 0 \text{ in }L^{2}_{loc}(0, T; L^{2}(M)),$$ which leads to $$u_{n} \to 0 \text{ in } L^{2}_{loc}(0, T; H^{\frac{\sigma}{2}}(M)).$$

We finish the proof by taking $t_{0} \in [0, T]$ such that $u_{n}(\cdot,t_{0}) \to 0$ in $H^{\frac{\sigma}{2}}(M)$ and $\{u_n\}$ solving the 
equation 
\begin{eqnarray*}
i \partial_{t}u_{n} + \Lambda^{\sigma}_{g}u_{n} = f_{n},
\end{eqnarray*}
with initial data $u_{n}(\cdot,t_{0}) \in H^{\frac{\sigma}{2}}(M)$. By uniqueness of solution and the conservation of energy, we get that $$u_{n} \to  0\text{ in }L^{\infty}(0, T; H^{\frac{\sigma}{2}}(M)),$$
which concludes
the proof of Proposition \ref{propagacaocompacidade}.
\end{proof}
The next result links the microlocal defect measures constructed in Proposition \ref{propagacaocompacidade}  with its analogue for the nonlinear problem.

\begin{proposition}\label{ppp}
Let $\{u_{n}\}$ be a sequence in $C([0, T]; H^{\frac{\sigma}{2}}(M))$ of solutions to
\begin{equation}
i \partial_{t}u_{n} + \Lambda_{g}^{\sigma}u_{n} + P'(|u_{n}|^{2})u_{n} = 0,
\end{equation}
such that $u_{n}(0)$ is weakly convergent to zero in $H^{\frac{\sigma}{2}}(M)$. Then $$Q(|u_{n}|^{2})u_{n} = P'(|u_{n}|^{2})u_{n} - P'(0)u_{n},$$ strongly converges to zero in $L^{2}(0, T; H^{\frac{\sigma}{2}}(M))$.
\end{proposition}
\begin{proof}
Observe that as $P'$ is a polynomial function we can estimate the norm of $Q(|u_{n}|^{2})u_{n}$ in $L^{2}(0, T; H^{\frac{\sigma}{2}}(M))$ by the terms of the form 
\begin{equation}
I_{k} := \int_{0}^{T}{\|u_{n}(t)\|_{L^{\infty}}^{2k}\|u_{n}(t)\|_{H^{\frac{\sigma}{2}}}^{2}}dt, \quad k \in \mathbb{N}.
\end{equation}
Since the norm of $u_{n}$ is  uniformly bounded in $H^{\frac{\sigma}{2}}(M)$, by conservation of energy, we get that
\begin{eqnarray}\label{aa}
I_{k} \leq C\int_{0}^{T}{\|u_{n}(t)\|_{L^{\infty}}^{2k}}dt.
\end{eqnarray}

To bound the right hand side of \eqref{aa}, note that if $(p,q)$ satisfies \eqref{St-Admssi} and $\infty > p > 2k$ we can get $\alpha' \in \left]\frac{2}{q}, \infty\right[$. Thus,
by using Sobolev embeddings and interpolation, yields that
\begin{eqnarray*}
\|u_{n}(t)\|_{L^{\infty}} \leq C\|u_{n}(t)\|_{W^{\alpha', q}} \leq C\|u_{n}(t)\|_{L^{q}}^{1- \theta}\|u_{n}(t)\|_{W^{\alpha, q}}^{\theta},
\end{eqnarray*} 
with $\alpha = s - \frac{1}{p} > \frac{d}{q}$ (see also \eqref{def-Yt}) and some $\theta \in [0, 1]$. Since $2k\theta < p$, by using Holder's inequality, we get that 
\begin{eqnarray}\label{aaa}
\int_{0}^{T}{\|u_{n}(t)\|_{L^{\infty}}^{2k}}dt \leq \left(\int_{0}^{T}{\|u_{n}(t)\|_{L^{q}}^{\beta}}dt \right)^{\delta} \left(\int_{0}^{T}{\|u_{n}(t)\|_{W^{\alpha, q}}^{p}}dt \right)^{\frac{2k\theta}{p}},
\end{eqnarray}
with $\beta, \delta > 0$. 

Observe that the second term of \eqref{aaa} is uniformly bounded in $n \in \mathbb{N}$ in view of Remark \ref{Remarks} item ii. On the other hand, the first term goes to zero, as $n \to \infty$, since we can estimate the norm of 
$u_{n}(t)$ in $L^{q}(M)$ in terms of its norm in $L^{2}(M)$ and $H^{\frac{\sigma}{2}}(M)$. Observe that, additionally, the norm of $u_{n}(t)$ is uniformly bounded by the energy estimate. Moreover, by the conservation of the mass, $\|u_{n}(t)\|_{L^{2}} = \|u_{n}(0)\|_{L^{2}}$, where $u_{n}(0) \to 0$ strongly in $L^{2}(M)$, since $u_{n}(0)\to 0$ weakly in $H^{\frac{\sigma}{2}}(M)$, which is compactly embedded in $L^{2}(M)$ by the Rellich Theorem. Summarizing, we have that
\begin{eqnarray*}
Q(|u_{n}|^{2})u_{n} \to 0, \quad \text{ as } k \to \infty,
\end{eqnarray*}
strongly in $L^{2}(0, T; H^{\frac{\sigma}{2}}(M))$.

This completes the proof, so Proposition \ref{ppp} is achieved.
\end{proof}

\begin{rem} We remark that, in terms of microlocal defect measures, Proposition \ref{ppp} asserts that the measure associated with the nonlinear problem and the one associated with the linear problem are equals. 
\end{rem}

\section{Stabilization for the fractional Schr\"odinger equation}\label{sec4}

The main goal of this section is to prove Theorem \ref{teo-stab}. First, observe that solution $u \in C(\mathbb{R}_{+}; H^{\frac{\sigma}{2}}(M))$  of \eqref{prob-diss1} obtained in Theorem \ref{Prob-diss-Teo} satisfies the semigroup property. Thus, in view of the energy identity \eqref{ene},  Theorem \ref{teo-stab} is a consequence of the following \textit{observability inequality}. 
\begin{proposition}
Under assumptions of Theorem \ref{teo-stab}, for every $T > 0$ there exists a 
constant $C > 0$ such that the inequality 
\begin{eqnarray}\label{observabilitty}
E(0) \leq C \int_{0}^{T}{\|(1 - \Delta_{g})^{-\frac{\sigma}{2}}a(x)\partial_{t}u(t) \|_{L^{2}(M)}^{2}}dt
\end{eqnarray}
holds for every solution $u=u(x,t)$ of the damped system \eqref{prob-diss} with the initial data $u_0$ satisfying $$\|u_{0}\|_{H^{\frac{\sigma}{2}}(M)} \leq R_{0},$$ for some $R_{0} > 0$.
\end{proposition}

\begin{proof}
We argue by contradiction. Suppose that \eqref{observabilitty} is not true, then there exists a sequence $\{u_n\}:=\{u_{n}\}_{n\in\mathbb{N}}$ of solutions to \eqref{prob-diss} such that 
\begin{eqnarray}\label{cont-1}
\|u_{n}(0)\|_{H^{\frac{\sigma}{2}}(M)} \leq R_{0}
\end{eqnarray}
and 
\begin{eqnarray}\label{cont-2}
\int_{0}^{T}{\|(1 - \Delta_{g})^{-\frac{\sigma}{2}}a(x)\partial_{t}u_{n}(t) \|_{L^{2}(M)}^{2}}dt \leq \frac{1}{n}E_{n}(0).
\end{eqnarray}
Denote $\alpha_{n} = \left(E_{n}(0) \right)^{\frac{1}{2}}$, . Thanks to \eqref{cont-1}, we have that $(\alpha_{n})$ is bounded, thus, we can extract a subsequence, which we shall not relabel,  such that $$\alpha_{n}\longrightarrow\alpha.$$ 
We split the analysis into two cases:  $\alpha > 0$ and $\alpha = 0$.

\vspace{0.1cm}

\noindent\textit{First case:} $\alpha_{n}\longrightarrow\alpha>0$. 

\vspace{0.2cm}
By using the fact that the energy decreases, we have that $(u_{n})$ is bounded in 
$L^{\infty}(0, T; H^{\frac{\sigma}{2}}(M))$. Additionally, 
\begin{eqnarray*}
\partial_{t}u_{n} = J^{-1}\left(i \Lambda_{g}^{\frac{\sigma}{2}}u_{n} +i P'(|u_{n}|^{2}u_{n}) \right),
\end{eqnarray*}
where $J$ is defined in Theorem \ref{Prob-diss-Teo}, hence $\{u_n\}$ is bounded in $C^{1}([0, T]; \mathcal{D}'(M))$. Therefore, we find a subsequence of $\{u_n\}$, still denoted by the same index, such that, 
for every $t \in [0, T]$, $$u_{n}(t) \rightharpoonup u(t),$$ for some $u \in C_{w}([0, T]; H^{\frac{\sigma}{2}}(M))$. Taking into account \eqref{cont-2} and passing to the limit in the system \eqref{prob-diss} we get that $u$ satisfies
\begin{equation}\label{prob-lim-est}
\begin{cases}
i \partial_{t} u  + \Lambda^{\sigma}_{g}u  + P'(|u|^{2})u= 0, &\text{ on } M\times ] 0, T[,
\\\partial_{t}u= 0,& \text{ on } \omega\times ] 0, T[.
\end{cases}
\end{equation}
Moreover, by using Remark \ref{Remarks} item ii., 
if $u \in L^{\infty}(0, T; H^{\frac{\sigma}{2}}(M))$ is a solution of \eqref{prob-lim-est} then $u \in L^{p}(0, T; L^{\infty}(M))$,
for every $p < \infty$, and the nonlinear term $P'(|u|^{2})u$ lies in $L^{2}(0, T; H^{\frac{\sigma}{2}}(M))$. 

By Duhamel formula and the uniqueness of the Cauchy problem for the fractional Schr\"odinger equation we conclude that $u \in C([0, T]; H^{\frac{\sigma}{2}}(M))$. Note that, taking into account assumption  $(\mathcal{B})$ and Theorem \ref{teoUCP}, the unique solution of \eqref{prob-lim-est} is the trivial one $u \equiv 0$. Hence, $\{u_{n}\}$ weakly converges to $0$.  A consequence of this convergence is that 
$$
Q(|u_{n}|^{2})u_{n} = P'(|u_{n}|^{2})u_{n} - P'(0)u_{n}\longrightarrow 0,
$$
strongly in $L^{2}(0, T; H^{\frac{\sigma}{2}}(M))$. On the other hand, by the contradiction hypothesis, $$a(x)(1 - \Delta_{g})^{- \frac{\sigma}{2}}a(x)\partial_{t}u_{n}\longrightarrow 0,$$  strongly in $L^{2}(0, T; H^{\frac{\sigma}{2}}(M))$. Summarizing, we get
\begin{eqnarray*}
&& \sup_{t \in [0, T]}\|u_{n}(t)\|_{L^{2}(M)} \longrightarrow 0, \\
&& \sup_{t \in [0, T]}\|u_{n}(t)\|_{H^{\frac{\sigma}{2}}(M)} \leq C\\
\end{eqnarray*}
and
\begin{eqnarray*}
&& i\partial_{t}u_{n} + \Lambda_{g}^{\sigma}u_{n} + P'(0)u_{n} \longrightarrow 0,
\end{eqnarray*}
with the last convergence in $L^{2}(0, T; H^{\frac{\sigma}{2}}(M))$. Therefore, we are in the conditions of Proposition \ref{propagacaocompacidade} and, then we can conclude that $$u_{n}\longrightarrow 0,$$ in $L^{\infty}(0, T; H^{\frac{\sigma}{2}}(M))$, which is a contradiction with the hypothesis $\alpha > 0$.

\vspace{0.1cm}

\noindent\textit{Second case:} $\alpha_{n}\longrightarrow\alpha=0$. 

\vspace{0.2cm}
Set $\{v_n\}_{n\in\mathbb{N}}:=\{v_{n}\} = \left\{ \frac{u_{n}}{\alpha_{n}}\right\}$. The new function satisfies 
 \begin{equation}\label{prob-cont-est-alpha0}
i \partial_{t} v_{n}  + \Lambda^{\sigma}_{g}v_{n}  + P'(|\alpha_{n}v_{n}|^{2})v_{n} - a(x)(1 - \Delta_{g})^{\frac{\sigma}{2}}a(x)\partial_{t}v_{n} =0, \quad \text { on } M\times] 0, T[ 
\end{equation}
and the estimate 
\begin{eqnarray}
\int_{0}^{T}{\|(1 - \Delta_{g})^{-\frac{\sigma}{4}}a(x)\partial_{t}v_{n}\|_{L^{2}(M)}^{2}}dt \leq \frac{1}{n}
\end{eqnarray}
while $\left\|v_{n}(0)\right\|_{H^{\frac{\sigma}{2}}(M)} \simeq 1$. Note that $\{v_{n}\}$ is 
bounded in $L^{\infty}(0, T; H^{\frac{\sigma}{2}}(M))$ and, by \eqref{prob-cont-est-alpha0}, we have $\{v_{n}\}$
bounded in $C^{1}([0, T]; \mathcal{D}'(M))$.
Hence, it admits a subsequence, still denoted by $\{v_{n}\}$, such that, for every $t \in [0, T]$, $$v_{n}(t) \rightharpoonup v(t),$$ for some $v \in C_{w}([0, T]; H^{\frac{\sigma}{2}}(M))$. 

Next, applying an estimate like presented in Remark \ref{Remarks} item ii. to \eqref{prob-cont-est-alpha0}, we get 
\begin{eqnarray*}
\|v_{n}\|_{Y_{T}} \leq C(1 + \alpha_{n}^{\mu}\|v_{n}\|_{Y_{T}}^{\beta}),
\end{eqnarray*}
where $\mu = 2 d^{\circ}P'$ and $\beta = 2d^{\circ}P - 1 $. Observe that  $\|v_{n}\|_{Y_{T}}$ depends continuously on $T$ and is bounded for $T = 0$. To conclude the proof we need to apply a boot-strap argument. Note that $F(t) := \|v_{n}\|_{Y_{t}}$ 
(fixing $n\in \mathbb{N}$) satisfies
\begin{eqnarray*}
F(t) \leq C(1 + \alpha_{n}^{\mu}F(t)^{\beta}).
\end{eqnarray*}
Now, we are in the condition of \cite[Lemma 2.2]{BaGe}, since $C \alpha_{n} \to 0$ and $F(0) = \|v_{0n}\|_{H^{\frac{\sigma}{2}}(M)}$ is  uniformly bounded with respect to $n \in \mathbb{N}$, or equivalently,
\begin{eqnarray*}
F(0) = \|v_{0n}\|_{H^{\frac{\sigma}{2}}(M)} \leq C, \forall n \in \mathbb{N}.
\end{eqnarray*}
Thus, taking $n \geq n_{0}$, for some $n_{0}$ large enough, we get that
\begin{eqnarray*}
F(0) \leq \frac{1}{\beta \alpha_{n}^{\frac{\mu}{\beta - 1}}},
\end{eqnarray*}
which yields
\begin{eqnarray*}
\|v_{n}\|_{Y_{T}} \leq C, 
\end{eqnarray*}
that is, $\|v_{n}\|_{Y_{T}}$ is uniformly bounded in $n \in \mathbb{N}$ and, therefore by Proposition \ref{ppp}, $$Q(|\alpha_{n}v_{n}|^{2})v_{n}\longrightarrow0 \text{ in  }L^{2}([0, T]; H^{\frac{\sigma}{2}}(M)).$$ Then, $v$ satisfies
\begin{equation*}
\begin{cases}
i \partial_{t}v + \Lambda^{\sigma}_{g}v + P'(0)v= 0, &\text{ on } M\times ] 0, T[,
\\\partial_{t}v= 0,& \text{ on } \omega\times ] 0, T[.
\end{cases}
\end{equation*}
Thanks to the uniqueness of the Cauchy problem for the linear fractional Schr\"odinger equation, we have that $v \in C([0, T]; H^{\frac{\sigma}{2}}(M))$, and so, $v = 0$ by unique continuation property. 

Finally, at this point, we can argue as in the first case: Applying Proposition \ref{propagacaocompacidade} we have that $v_{n} \to 0$ strongly in $C([0, T]; H^{\frac{\sigma}{2}}(M))$, which is a contradiction since $\|v_{0n}\|_{H^{\frac{\sigma}{2}}(M)}\simeq 1$, thus achieving the proof.
\end{proof}

\section{Exact controllability for the fractional Schr\"odinger equation }\label{sec5}

In this section we prove the exact controllability at level $H^{s}(M)$, with $s \geq \frac{\sigma}{2}$, for the nonlinear fractional Schr\"odinger equation, namely
\begin{equation}\label{Ncontrolsys} 
\begin{cases}
i \partial_{t} u  + \Lambda^{\sigma}_{g}u  + P'(|u|^{2})u = h,&   \text{ on } M\times ] 0, T[, \\
u(x,0)=u_0(x),& x\in M.
\end{cases}
\end{equation}
To achieve that issue, we use the classical duality approach \cite{DolRus1977,lions1}, which reduces the controllability problem associated to system \eqref{Ncontrolsys} to prove an observability inequality by using the so-called “Compactness-Uniqueness Argument” due to J.-L. Lions \cite{lions1} for solutions of the following linear system
\begin{equation}\label{controlsys} 
\begin{cases}
i \partial_{t} u  + \Lambda^{\sigma}_{g}u  = h,&   \text{ on } M\times ] 0, T[, \\
u(x,0)=u_0(x),& x\in M.
\end{cases}
\end{equation}
Finally, with the linear control problem in hand, the idea is to consider the control  operator associated to the nonlinear problem as a perturbation of the control operator associated  to the linear one.

\subsection{Controllability for the linear fractional Schr\"odinger equation}
The goal here is to prove the controllability for the following linear system,
\begin{equation}\label{Lcontrolsys} 
\begin{cases}
i \partial_{t} u  + \Lambda^{\sigma}_{g}u  = \tilde{h},&   \text{ on } M\times ] 0, T[, \\
u(x,0)=u_0(x),& x\in M.
\end{cases}
\end{equation}

Note that to show the results of this subsection we only need the set $\omega \subset M$, where the control is effective,  satisfying the assumption $(\mathcal{A})$. It is important to note that the compactness-uniqueness argument reduces the problem to prove a unique continuation property (UCP) for solutions of \eqref{Lcontrolsys}, however,  (UCP) required in this case is derived from the properties of second order elliptic operators due to H\"ormander \cite{Hormander}. 

We are now in a position to prove the observability inequality associated to \eqref{Lcontrolsys}. The result can be read as follows.

\begin{proposition}\label{PropControlLin}
Let $\omega \subset M$ be an open set 
satisfying the assumption $(\mathcal{A})$ and $a \in C^{\infty}(M)$ a real valued function such that $a \equiv 1$ 
on $\omega$. Then, for every $T> 0$, there exists a constant $C = C(T) > 0$ such that, for every solution of system \eqref{Lcontrolsys},
with $s \geq \frac{\sigma}{2}$, $\tilde{h} = 0$ and initial data $u(0)=u_{0} \in H^{-s}(M)$, we have the following inequality
\begin{eqnarray} \label{obs-linprob}
\|u_{0}\|_{H^{-s}(M)}^{2} \leq C \int_{0}^{T}{\|a u(t)\|_{H^{-s}(M)}^{2}}dt.
\end{eqnarray}
\end{proposition}
\begin{proof}
We split the proof into three steps, first we will prove an auxiliary inequality.

\vspace{0.2cm}

\noindent \textbf{Step 1.} We start by proving the following
 estimate 
 \begin{eqnarray} \label{obs-relaxed}
 \|u_{0}\|_{	H^{-s}(M)}^{2} \leq C\left(\int_{0}^{T}{\|a u(t)\|_{H^{-s}(M)}^{2}}dt + \| (1 - \Delta_{g})^{- \frac{s}{2} - \frac{\sigma}{4}}u_{0} \|_{H^{-s}(M)}^{2} \right),
 \end{eqnarray}
 for $s \geq \frac{\sigma}{2}$ and $u$ solution of \eqref{Lcontrolsys}. 
 
 \vspace{0.2cm}

 We argue by contradiction, suppose that \eqref{obs-relaxed} does not occur.  Then,  there exists a sequence $\{u^{n}\}:=\{u^{n}\}_{n\in\mathbb{N}}$ of solution of \eqref{Lcontrolsys}  satisfying
 \begin{equation}
 \|u_{0}^{n}\|_{H^{-s}(M)} = 1, \label{CL-1}
 \end{equation}
 for all $n \in \mathbb{N}$. Additionally, we have, when $n \to \infty$, that
 \begin{eqnarray}
 && a u^{n} \rightarrow 0 \quad \hbox{ in } \quad L^{2}(0, T; H^{-s}(M))  \label{CL-2}
 \end{eqnarray}
 and 
 \begin{eqnarray}
&& \|(1 - \Delta_{g})^{- \frac{s}{2} - \frac{\sigma}{4}}u_{0}^{n}\|_{H^{-s}(M)} \rightarrow 0. \label{CL-3}
 \end{eqnarray}

Note that $u^{n}$ is bounded in $L^{\infty}(0, T; H^{-s}(M))$. By using the first equation of \eqref{Lcontrolsys}, with $\tilde{h} = 0$, we have $i\partial_{t}u^{n} =- \Lambda^{\sigma}_{g}u^{n}$ and, consequently, $\partial_{t}u^{n}$ is bounded in 
$L^{\infty}(0, T; H^{-s - \sigma}(M))$. So, we may extract a subsequence (still denoted $\{u^n\}$) such that,
 for every $t \in [0, T]$, $$u^{n}(t) \rightharpoonup u(t) \text{ weakly in } H^{-s}(M),$$ 
  for some $u \in L^{\infty}(0, T; H^{-s}(M))$. Thanks to \eqref{CL-3} the sequence of initial data converges to $0$, so we can conclude, by passing the limit on \eqref{Lcontrolsys}, that $u \equiv 0$. 
  
  Now, let us introduce $w_{n}:= (1 - \Delta_{g})^{- \frac{s}{2} - \frac{\sigma}{4}}u^{n}$, $n \in \mathbb{N}$, so  $\{w_{n}\}$ satisfies the equation \eqref{Lcontrolsys} and $w_{n}(t)$ converges weakly to $0$ in 
  $H^{\frac{\sigma}{2}}(M)$, hence strongly to zero in $L^{2}(M)$, for every $t \in [0, T]$. 
 Note that by definition of $\{w_n\}$, we have the following
  \begin{eqnarray*}
  a w_{n} &=& a (1 - \Delta_{g})^{- \frac{s}{2} - \frac{\sigma}{4}}u^{n} \nonumber \\
  & = & (1 - \Delta_{g})^{- \frac{s}{2} - \frac{\sigma}{4}}(a u^{n})  + [a, (1 - \Delta_{g})^{- \frac{s}{2} - \frac{\sigma}{4}}](1 - \Delta_{g})^{\frac{s}{2} + \frac{\sigma}{4}}w_{n} \nonumber \\
  & := & I_{1} + I_{2} .
  \end{eqnarray*}   
By using \eqref{CL-2} and that $(1 - \Delta_{g})^{-\frac{s}{2} - \frac{\sigma}{4}}$ 
  is a continuous linear operator, it follows that $$I_{1} \longrightarrow 0 \quad \text{ in } \quad L^{2}(0, T; H^{\frac{\sigma}{2}}(M)).$$
 On the other hand,  since $[a, (1 - \Delta_{g})^{-\frac{s}{2} - \frac{\sigma}{4}}](1 - \Delta_{g})^{\frac{s}{2} + \frac{\sigma}{4}}$ is a pseudo-differential operator of order $-1$, we have that 
 $$I_{2} \longrightarrow 0 \quad \text{ in } \quad L^{2}(0, T; H^{\frac{\sigma}{2}}(M)).$$
 These convergences yields that 
 $$w_{n} \longrightarrow 0 \quad \text{ in } \quad L^{2}(0, T; H^{\frac{\sigma}{2}}(\omega)).$$ Therefore,  thanks to Proposition \ref{propagacaocompacidade}, we get 
  $w_{n} \rightarrow 0$ in $L^{\infty}(0, T; H^{\frac{\sigma}{2}}(M))$, which means that 
  $u^{n} \rightarrow 0$ in $L^{\infty}(0, T; H^{-s}(M))$, a contradiction with $\|u^{n}(0)\|_{H^{-s}(M)} = 1$. Thus, \eqref{obs-relaxed} holds.
  
\vspace{0.2cm}

\noindent  \textbf{Step 2. } Consider
  \begin{eqnarray*}
  \mathcal{N}(T) = \{u_{0} \in H^{-s}(M): a u = 0 \hbox{ on } ]0, T[ \times M\}.
  \end{eqnarray*}
We claim that $\mathcal{N}(T) = \{ 0 \}$. 

\vspace{0.2cm}

Indeed, first we note that assumption $(\mathcal{A})$ and  Proposition \ref{reg-propagacao-prop1} ensures that $\mathcal{N}(T) \subset C^{\infty}(M)$ 
  and  $u$ vanishes on $\omega$. Since the Laplace operator commutes with the 
  linear equation \eqref{Lcontrolsys}, with $\tilde{h} = 0$, we have 
  $$ i \partial_{t}(\Delta_{g}u) + \Lambda^{\sigma}_{g}(\Delta_{g}u) = 0$$
  and 
  $$ \Delta_{g}u(0) = \Delta_{g}u_{0},$$
  from which we can conclude that $\Delta_{g}(\mathcal{N}(T)) \subset \mathcal{N}(T)$. On the other hand, 
  by \eqref{obs-relaxed} we have that 
  \begin{eqnarray*}
  \|u_{0}\|_{H^{-s}(M)}^{2} \leq C\|(1 - \Delta_{g})^{ - \frac{s}{2} - \frac{\sigma}{4}}u_{0} \|_{H^{-s}(M)}^{2},
  \end{eqnarray*}
  which means that $\mathcal{N}(T)$ is finite dimensional subspace. Then if $\mathcal{N}(T) \neq \{ 0 \}$ there would exist $z_{0} \in \mathbb{C}$ and $u_{0} \in \mathcal{N}(T)$ such that 
  \begin{eqnarray*}
  \Delta_{g} u_{0} + z_{0}u_{0} = 0 \quad \text{ on } \quad\omega.
  \end{eqnarray*}
Therefore,  by unique continuation theorem for second order
elliptic operators (see  \cite[Theorem 17.2.6]{Hormander}), we conclude $u_{0} = 0$ on $M$, a contradiction. Thus, $\mathcal{N}(T)=\{0\}$.
  
  \vspace{0.2cm}

\noindent  \textbf{Step 3. } The last step is to remove the second term on the right hand side of \eqref{obs-relaxed}, that is to show \eqref{obs-linprob}.

  \vspace{0.2cm}

In fact, let us again argue by contradiction. Suppose that \eqref{obs-linprob} does not hold, thus there exists a sequence $\{u_{0}^{n}\}=\{u_{0}^{n}\}_{n\in\mathbb{N}}$ 
  such that 
  \begin{equation}\label{id}
  \|u_{0}^{n}\|_{H^{-s}(M)} = 1
  \end{equation}
  and
    \begin{equation}\label{id1}
\int_{0}^{T}{ \|a u_{n}(t)\|_{H^{-s}(M)}^{2}}dt \rightarrow 0.
  \end{equation}
Since $(1 - \Delta_{g})^{-\frac{s}{2} - \frac{\sigma}{4}}$ is a compact operator 
  on $H^{-s}(M)$ we may extract a subsequence, still denoted by the same index,  such that 
  \begin{eqnarray}\label{id2}
  (1 - \Delta_{g})^{-\frac{s}{2} - \frac{\sigma}{4}}u_{0}^{n} \rightarrow (1 - \Delta_{g})^{-\frac{s}{2} - \frac{\sigma}{4}}u_{0} \quad \hbox{ in } H^{-s}(M),
  \end{eqnarray}
  for some $u_{0} \in H^{-s}(M)$. Consider 
  $u$ solution of \eqref{Lcontrolsys} with initial data 
  $u_{0}$. By using \eqref{id1} we get that 
  \begin{eqnarray*}
   \int_{0}^{T}{\|a u(t)\|_{H^{-s}(M)}^{2}}dt =0,
  \end{eqnarray*}
  that is, $a u = 0$ on $]0, T[ \times M$. Therefore, it follows that $u_{0} \in \mathcal{N}(T)=\{0\}$.
  
   Finally, by Step 1, we have 
  \begin{eqnarray}\label{dd}
 \|u_{0}^{n}\|_{H^{-s}(M)}^{2} & \leq & \int_{0}^{T}{ \|a u_{n}(t)\|_{H^{-s}(M)}^{2}}dt +  C\|(1 - \Delta_{g})^{- \frac{s}{2} - \frac{\sigma}{4}}u_{0}^{n}\|_{H^{-s}(M)}^{2}.
  \end{eqnarray}
This inequality combined with \eqref{id2} give us that
   \begin{eqnarray*}
 1=\|u_{0}^{n}\|_{H^{-s}(M)}^{2} & \to & 0,
  \end{eqnarray*}
   which is a contradiction. Therefore, Proposition \ref{PropControlLin} is proved.
\end{proof}

Let us now give a first answer for the controllability problem for the linear problem proposed in this article.

\begin{teo}\label{th_lin} Under the assumptions of Proposition \ref{PropControlLin}, 
for every initial data $u_{0} \in H^{s}(M)$, with $s \geq \frac{\sigma}{2}$ and every $T > 0$, there exists a control $h \in C(\mathbb{R}_{+}; H^{s}(M))$, with support 
in $\mathbb{R}_{+} \times \omega$ such that the unique  solution of  \eqref{Lcontrolsys} with 
\begin{equation}
\tilde{h}=\begin{cases}
h,&\text{ if } 0 \leq t \leq T\\
0,&\text{ otherwise }
\end{cases}
\end{equation}
 satisfies $u(\cdot,t) = 0$  for $t \geq T.$
\end{teo}

\begin{proof}
By a compactness argument, there exists an open set $\omega'$ satisfying assumption $(\mathcal{A})$ such that $\overline{\omega}^{\prime} \subset \omega$. Let us consider $\varphi \in C^{\infty}(M)$ 
a real valued function supported in $\omega$ such that $\varphi \equiv 1$ on $\omega'$. Consider the systems 
\begin{equation}\label{dupla}
\begin{cases}
 i\partial_{t}u + \Lambda_{g}^{\sigma}u = h \in L^{1}(0, T; H^{s}(M)),\\
  u(T) = 0 
  \end{cases}
\end{equation}
and
\begin{equation}\label{duplaa}
\begin{cases}
i \partial_{t}v + \Lambda^{\sigma}_{g}v = 0,\\
 v(0) = v_{0} \in H^{-s}(M).
\end{cases}
\end{equation}
Multiplying \eqref{dupla} by $\overline{v}$ and integrating by parts, we get that
\begin{eqnarray}\label{duality}
\left<-iu_{0}, v_{0} \right> = \int_{0}^{T}{(h, v)_{L^{2}(M)}}dt, 
\end{eqnarray}
where $u_{0}=u(0)$ and $\left<\cdot,\cdot\right>$ is the duality between $H^{s}(M)$ and $H^{-s}(M)$.

Now consider  the following  continuous map 
$$\Gamma: H^{-s}(M) \rightarrow H^{s}(M),$$ defined by 
$\Gamma(v_{0}) = -iu_{0}$, with  the choice
\begin{eqnarray*}
h = Av = a(x)(1 - \Delta_{g})^{-s}a(x)v(x,t).
\end{eqnarray*}
By the duality relation \eqref{duality} and definition of $\Gamma$, we get that 
\begin{eqnarray*}
\left< \Gamma(v_{0}), v_{0}\right>& = & \int_{0}^{T}{(Av, v)_{L^{2}(M)}} 
\\
& = & \int_{0}^{T}{\|Bv(t)\|_{L^{2}(M)}^{2}}dt,
\end{eqnarray*}
where we denoted $Bv(t) = (1 - \Delta_{g})^{-\frac{s}{2}}a(x)v(t, x)$. In this way, 
\begin{eqnarray*}
\left<\Gamma(v_{0}), v_{0} \right> = \int_{0}^{T}{\|a v(t)\|_{H^{-s}(M)}^{2}}dt.
\end{eqnarray*}
Therefore, operator $\Gamma$ is self-adjoint and satisfies $$\|\Gamma(v_{0})\|_{H^{s}(M)} \geq C\|v_{0}\|_{H^{-s}(M)},$$ by virtue of the observability inequality \eqref{obs-linprob}. It, therefore, defines an isomorphism from $H^{-s}(M)$ to 
$H^{s}(M)$, which completes the proof of Theorem \ref{th_lin}.
\end{proof}

\subsection{Controllability for the nonlinear fractional Schr\"odinger equation}
We finish this section showing Theorem \ref{main1}. 
\begin{proof}[Proof of Theorem \ref{main1}]
We consider the following two systems 
\begin{equation}\label{dupla1}
\begin{cases}
 i \partial_{t}\phi + \Lambda^{\sigma}_{g}\phi + P'(0)\phi = 0, \\
\phi(0) = \phi_{0} \in H^{-s}
\end{cases}
\end{equation}
 and
 \begin{equation}\label{Sys-Cont-1}
\begin{cases}
i \partial_{t}u + \Lambda^{\sigma}_{g}u + P'(|u|^{2})u = A\phi, \\
u(T) = 0,
\end{cases}
\end{equation}
 where $A\phi := a(x)(1 - \Delta)^{-s}a(x)\phi$, for $s \geq \frac{\sigma}{2}$. Define, now, the following operator
 \begin{eqnarray*}
 L: H^{-s}(M)& \rightarrow &H^{s}(M)\\
\phi_{0} &\mapsto &L\phi_{0} = u_{0},
 \end{eqnarray*}
with $u_0=u(0)$.
 
\vspace{0.2cm}

\noindent\textbf{Claim 1.} $L$ is onto a small neighborhood of the origin of $H^{s}(M)$. 

\vspace{0.2cm}
Indeed,  first split $u$ into $u = v + \psi$ where $v$ satisfies 
\begin{equation}\label{Sys-v-1}
\begin{cases}
i \partial_{t}v + \Lambda^{\sigma}_{g}v + P'(0)v - Q(|u|^{2})u = 0, \\
v(T) = 0
\end{cases}
\end{equation}
 and $\psi$  satisfies 
 \begin{equation}\label{dupla2}
\begin{cases}
i \partial_{t}\psi + \Lambda^{\sigma}_{g}\psi + P'(0)\psi  = A\phi, \\
 \psi(T) = 0.
\end{cases}
\end{equation}
 
 Note that $u, v$ and $\psi$ belongs to $C([0, T]; H^{s}(M))$  and $u(0) = v(0) + \psi(0)$, that is,  
 \begin{eqnarray*}
 L\phi_{0} = K\phi_{0} + S\phi_{0} =u_0,
 \end{eqnarray*}
where $$K\phi_{0} := v(0).$$
Due to $S: H^{-s}(M) \rightarrow H^{s}(M)$ be the linear control isomorphism between $H^{-s}(M)$ and $H^{s}(M)$ (see Theorem \ref{th_lin}), we have that $L\phi_{0} = u_{0}$ 
 is equivalent to $\phi_{0} = - S^{-1}K\phi_{0} + S^{-1}u_{0}$.  Let us consider the following operator $$B: H^{-s}(M) \rightarrow H^{-s}(M)$$ defined by $$B\phi_{0} = - S^{-1}K\phi_{0} + S^{-1}u_{0}.$$ We get that $L\phi_{0} = u_{0}$ if 
 $\phi_{0}$ is a fixed point of $B$. Therefore, Claim 1 is equivalent to find a fixed point for the operator $B$ near the origin of $H^{-s}(M)$. Let us prove it. To do this, 
 we may assume $T < 1$. Since $S$ is a continuous linear operator, we have
 \begin{eqnarray*}
 \|B\phi_{0}\|_{H^{-s}(M)} &\leq& C\left(\|K\phi_{0}\|_{H^{s}(M)} + \|u_{0}\|_{H^{s}(M)} \right)\\
 & \leq & C \left(\|v(0)\|_{H^{s}(M)} + \|u_{0}\|_{H^{s}(M)} \right).
 \end{eqnarray*} 
By Strichartz estimates at level $H^{s}(M)$  applied to the system \eqref{Sys-v-1}, it follows that
 \begin{eqnarray*}
 \|B\phi_{0}\|_{H^{-s}(M)} \leq C\left(\int_{0}^{1}{\|Q(|u|^{2})u(t)\|_{H^{s}(M)}}dt + \|u_{0}\|_{H^{s}(M)} \right).
 \end{eqnarray*}
From estimate \eqref{contd-1} we deduce the following
 \begin{eqnarray*}
 \|B\phi_{0}\|_{H^{-s}(M)} \leq C\left(\|u\|_{Y_{1}}^{\beta} + \|u_{0}\|_{H^{s}(M)} \right). 
 \end{eqnarray*}
 On the other hand, taking $$\|\phi_{0}\|_{H^{-s}(M)} \leq R,$$ with $R \leq 1$ small enough, Remark \ref{Remarks} applied to the system \eqref{Sys-Cont-1}, ensures that
 \begin{eqnarray*}
 \|u\|_{Y_{1}} & \leq & C \sup_{0 \leq t \leq 1}\|A\phi\|_{H^{s}(M)} \\
 & \leq & C\|\phi_{0}\|_{H^{-s}(M)}.
 \end{eqnarray*}
Therefore, 
 \begin{eqnarray*}
 \|B\phi_{0}\|_{H^{-s}(M)} & \leq & C\left[\left(\sup_{0 \leq t \leq 1}\|A\phi\|_{H^{s}(M)} \right)^{\beta} + \|u_{0}\|_{H^{s}(M)} \right] \\
 & \leq & C \left(\|\phi_{0}\|_{H^{-s}(M)}^{\beta} + \|u_{0}\|_{H^{s}(M)}  \right).
 \end{eqnarray*}
 Choosing  $u_{0}$ such that $$\|u_{0}\|_{H^{s}(M)} \leq \frac{R}{2C},$$ we have 
 $$\|B\phi_{0} \|_{H^{-s}(M)} \leq R,$$ which means that operator $B$ reproduces the ball $B_{R}$ in $H^{-s}(M)$. 
 
 It remains to prove that $B$ is a contraction, so, the proof of Claim 1 is achieved proving the following: 
 
 \vspace{0.2cm}

\noindent\textbf{Claim 2.} $B$ is a contraction on a small ball $B_{R}$ of $H^{-s}(M)$.

\vspace{0.2cm}

 Indeed,  let us consider the systems,
 \begin{equation*}
\begin{cases}
 i\partial_{t}(v_{1} - v_{2}) + \Lambda^{\sigma}_{g}(v_{1} - v_{2}) + P'(0)(v_{1} - v_{2}) = Q(|u_{1}|^{2})u_{1} - Q(|u_{2}|^{2})u_{2}, \\
(v_{1} - v_{2})(T) = 0
\end{cases}
\end{equation*}
 and
 \begin{equation*}
\begin{cases}
i\partial_{t}(u_{1} - u_{2}) + \Lambda^{\sigma}_{g}(u_{1} - u_{2}) = P'(|u_{2}|^{2})u_{2} - P'(|u_{1}|^{2})u_{1} + A(\phi^{1}) - A(\phi^{2}) , \\
(u_{1} - u_{2})(T) = 0.
\end{cases}
\end{equation*}
 Hence, 
 \begin{eqnarray} \label{control1}
 \begin{split}
 \|B(\phi^{1}_{0}) - B(\phi_{0}^{2})\|_{H^{-s}(M)} &\leq C \|(v_{1} - v_{2})(0)\|_{H^{s}(M)} \\
&\leq  C\left(\|u_{1}\|_{Y_{1}}^{\beta - 1} + \|u_{2}\|_{Y_{1}}^{\beta - 1} \right)\|u_{1} - u_{2}\|_{Y_{1}}.
\end{split}
 \end{eqnarray}
 On the other hand, a similar estimate applied to $u_{1} - u_{2}$ yields that
 \begin{eqnarray*}
 \|u_{1} - u_{2}\|_{Y_{1}} & \leq & C\left(\|u_{1}\|_{Y_{1}}^{\beta - 1} + \|u_{2}\|_{Y_{1}}^{\beta - 1} \right)\|u_{1} - u_{2}\|_{Y_{1}} + C\sup_{0 \leq t \leq 1}\|A(\phi^{1} - \phi^{2})\|_{H^{s}(M)} \\
 & \leq & C2R^{\beta - 1}\|u_{1} - u_{2}\|_{Y_{1}} + C\sup_{0 \leq t \leq 1}\|A(\phi^{1} - \phi^{2})\|_{H^{s}(M)}.
 \end{eqnarray*}
Taking $R$ small enough, we get
 \begin{eqnarray} \label{control2}
 \|u_{1} - u_{2}\|_{Y_{1}} \leq C \sup_{0 \leq t \leq 1}\|A(\phi^{1} - \phi^{2})(t)\|_{H^{s}(M)}.
 \end{eqnarray}
Finally, combining \eqref{control1} and \eqref{control2}, we conclude that
$$
\|B(\phi_{0}^{1}) - B(\phi_{0}^{2})\|_{H^{-s}(M)}  \leq CR^{\beta - 1}\sup_{0 \leq t \leq 1}\|A(\phi^{1} - \phi^{2})\|_{H^{s}(M)} \leq  CR^{\beta - 1}\|\phi_{0}^{1} - \phi_{0}^{2}\|_{H^{-s}(M)}.
$$
Thus, $B$ is a contraction on a small ball $B_{R}$ of $H^{-s}(M)$. Therefore, Claim 2 is proved and, consequently, Claim 1 holds. Thus, the proof of Theorem \ref{main1} is completed.
\end{proof}

\section{Concluding remarks and open problems}\label{sec6}
This manuscript deals with the nonlinear fractional Schrödinger equation (or generalized  nonlinear Schrödinger equation) with a forcing term $h=h(x,t)$ added to the equation as a control input, namely
\begin{equation}\label{prob-intro_aaaa}
\begin{cases}
i \partial_{t} u+\Lambda^{\sigma}_{g}u +P'(|u|^{2})u=h,&\quad \text{ on } M\times ]0, T[,
\\ u(x,0)=u_0(x),&\quad x\in M,
\end{cases}
\end{equation}
where $\sigma\in[2,\infty)$, $M$ is a compact Riemannian manifold with dimension $d<[\sigma] +1$. To conclude our article, let us discuss some general aspects of this work on the control theory for the NLS-like equation and its variations on manifolds. 

As mentioned at the beginning of the introduction, the well-posedness and controllability of the NLS on manifolds has been extensively studied in recent years. Given the relevance of the physical models associated with the generalized Schr\"odinger equation it would be natural to consider the same kind of problems for the generalized models.   

The pioneer work related with well-posedness for the fractional Schrödinger and fractional
wave equations on compact manifolds without boundary is due to Dinh in \cite{Dinh}.  He gave the first step by establishing Strichartz estimates for this model. Precisely,  he showed  Strichartz estimates for the fractional Schrödinger equations on compact Riemannian manifolds without boundary $M$ endowed with a smooth bounded metric $g$. In fact, the result extends the well-known Strichartz estimate for the Schr\"odinger equation given in \cite{Burq-Gerard-Tz}.

With this nice result in hand, in our work, we establish controllability results for the solutions of \eqref{prob-intro_aaaa}. The result studied here extends the result presented in \cite{dehman-gerard-lebeau} in the following sense. Considering $\sigma=2$ we recovered the well known result for NLS proved by Dehman \textit{et al.} in \cite{dehman-gerard-lebeau}. Additionally, our results ensure a more general framework related with the controllability problem for the NLS-like equations (see, for instance, \cite{CaCa} for a simple case,  $\sigma=4$). In summary, in this manuscript, we are able to prove control results for the generalized Schr\"odinger equation on any compact Riemannian manifold of dimension $d<[\sigma] +1$ in which the Strichartz estimates guarantees the uniform well-posedness and the control operator acts on a region satisfying the \textit{Geometric Control Condition}.

Even in this general context, there are still open problems to be investigated and this work opens up numerous possibilities for study, such as:

\vspace{0.2cm}

\noindent\textbf{Problem $\mathcal{A}$}: Is it possible to prove  similar well-posedness and  control results for the fractional Schr\"odinger equation \eqref{prob-intro_aaaa} on manifolds of dimension $d \geq [\sigma] +1$? 
 
\vspace{0.1cm}

 In this case, it is necessary to mention that bilinear Strichartz estimates, similar to the ones proved in \cite{Burq2}, should be proved in order to establish the uniform well-posedness result and, consequently, to use a variation of the techniques employed in this work to achieve the answer for this question.

\vspace{0.1cm}

To finish our discussion, we would like to mention facts related with the  \textit{Geometric Control Condition (GCC)}, introduced at the beginning of the work. The exact controllability is known to be true when geometric control condition is realized for NLS, see for instance, Lebeau \cite{Lebeau}, but also for any open set $\omega$ of $\mathbb{T}^{n}$, $n\in\mathbb{N}$
, see Jaffard \cite{Jaffard} and Komornik and Loreti \cite{KoLo}. Additionally, the exact controllability holds also for general manifolds considering (GCC), see for instance, Laurent \cite{Laurent-siam}. Moreover,  in some geometrical settings, as it was shown in \cite{Macia2}, (GCC) is equivalent of the observability property. In what concerns the fractional model, recently in \cite{Zhu}, it was proved that in the bi-dimensional torus $\mathbb{T}^{2}$, with the flat topology, 
(GCC) is not only sufficient but also necessary to achieve the desirable controllability results for the fractional Schrodinger equation with $\sigma \in (1, 2)$.  Thus, taking advantage of the results presented in this paper together with the result showed in \cite{Zhu}, we raise the following questions concerning the fractional Schrodinger equation:

\vspace{0.2cm}

\noindent\textbf{Problem $\mathcal{B}$}: Are there relation between geometric control condition and the controllability results for the generalized Schr\"odinger system? Are there some geometrical settings that geometric control condition is equivalent to the observability property for the generalized Schr\"odinger system? 

\vspace{0.2cm}

Precisely, both questions can be summarized in the following issue: 

\vspace{0.2cm}

\noindent\textbf{Problem $\mathcal{C}$}: Could one prove control results for the nonlinear fractional Schr\"odinger equation only with assumption $\mathcal{A}$?

\subsection*{Acknowledgments} 
R. de A. Capistrano--Filho was supported by CNPq 408181/2018-4, CAPES-PRINT 88881.311964/2018-01, CAPES-MATHAMSUD 88881.520205/2020-01, MATHAMSUD 21-MATH-03 and Propesqi (UFPE). A. Pampu was partially supported by PNPD-CAPES grant number 88887.351630/2019-00.

\end{document}